\documentclass[reqno,a4paper,11pt]{amsart}
\usepackage{amsmath,amsthm,amssymb,mathrsfs}
\usepackage[
left = 2.25cm,
right = 2.25cm,
top = 2.2 cm,
bottom = 2.2 cm
]{geometry}
\usepackage{graphicx,xcolor,tikz,mathtools}
\usepackage{bm,stmaryrd}
\usepackage{amsaddr}
\usepackage{enumitem}
\usepackage[toc,page]{appendix}
\usepackage[
pdfencoding=auto,
colorlinks = true,
citecolor = blue,
linkcolor = blue,
anchorcolor = blue,
urlcolor = blue, 
filecolor=blue,
pdfkeywords={}
]{hyperref}
\usepackage{esint} 

\usepackage{cite}
\usepackage{todonotes}
\makeatletter
\define@key{todonotes}{Helmut}[]{%
	\setkeys{todonotes}{author=\textbf{Helmut},inline,color=red!10}}%
\define@key{todonotes}{Andrea}[]{%
	\setkeys{todonotes}{author=\textbf{Andrea},inline,color=green!20}}%
\define@key{todonotes}{Yadong}[]{%
	\setkeys{todonotes}{author=\textbf{Yadong},inline,color=cyan!20}}%
\makeatother

\theoremstyle{plain}
\newtheorem{theorem}{Theorem}[section]

\newtheorem{definition}[theorem]{Definition}
\newtheorem{lemma}[theorem]{Lemma}

\newtheorem{remark}[theorem]{Remark}
\theoremstyle{definition}

\numberwithin{equation}{section}

\def\sf{\mathbb S^{d-1}}

\newcommand{\bn}{\mathbf{n}}

\newcommand{\bH}{\mathbf{H}}

\renewcommand{\d}{\mathrm{d}}
\newcommand{\dx}{\,\d x}
\newcommand{\dt}{\,\d t}
\newcommand{\dxdt}{\,\d x \d t}

\newcommand{\ddt}{\frac{\d}{\d t}}

\newcommand{\ptial}[1]{ \partial_{#1} }
\newcommand{\pt}{\ptial{t}}
\newcommand{\onehalf}{\frac{1}{2}}

\newcommand{\norm}[1]{\left\Vert #1 \right \Vert}

\newcommand{\normmm}[1]{\left\vert #1 \right\vert}

\def\Div{\mathrm{div}\,}
\DeclareMathOperator*{\esssup}{\mathrm{esssup}}
\newcommand{\dist}{\mathrm{dist}}

\DeclarePairedDelimiter\ceil{\lceil}{\rceil}
\DeclarePairedDelimiter\floor{\lfloor}{\rfloor}
\def\xhi{\chi}

\def\Th#1{{T_h(#1)}}

\def\Hz{H^{1}({\mathbb T^d})}
\def\normh#1{\norm{#1}_{L^2({\mathbb T^d})}}

\def\bL{\mathbf L}

\def\Ts{{T^*}}
\def\tF{\widetilde{E}}
\def\wchi{\widehat\chi}

\def\tmmu{\normmm{\mu_t}_{\mathbb{S}^{d-1}}}
\def\om{{\mathbb T^d}}

\def\wtilde{\widetilde}
\def\R{\mathbb R}

\def\N{\mathbb N}
\def\R{\mathbb R}

\def\tw{w}

\def\Ld#1{\norm{#1}_{(L^\infty({\mathbb T^d})\cap H^1({\mathbb T^d}))'}}
\def\dual{(L^\infty({\mathbb T^d})\cap H^1({\mathbb T^d}))'}

\allowdisplaybreaks[4]

\begin{document}
	
	\title[De Giorgi solutions to MCF: minimizing movements]{De Giorgi varifold solutions to Mean Curvature Flow:\\ a minimizing movements approach}
    
	\author[T. Laux]{
		\small
		Tim Laux$^{\ast}$
	}

	\address{
		$^{\ast}$Institut f\"{u}r Mathematik \& Interdisziplin\"{a}res Zentrum f\"{u}r Wissenschaftliches Rechnen, Universit\"{a}t Heidelberg, Im Neuenheimer Feld 205, 69120 Heidelberg, Germany
	}
	\email{tim.laux@math.uni-heidelberg.de}
    
	\author[A. Poiatti]{
		\small
		Andrea Poiatti$^\dagger$
	}

	\address{
		  $^\dagger$Dipartimento di Scienze Matematiche, Fisiche e Informatiche, \\ Università degli Studi di Parma, 43124 Parma, Italy
	}
	\email{andrea.poiatti@unipr.it}


		\subjclass[2020]{ 53E10, 53A10, 49Q20, 28A75}
		\keywords{mean curvature flow, volume-preserving mean curvature flow, De Giorgi varifold weak solutions, minimizing movements, De Giorgi interpolation}	
	
\begin{abstract}
	We propose an alternative existence proof of global weak solutions to mean curvature flow and volume preserving mean curvature flow. 
    We prove for the first time for a minimizing movements scheme the unconditional convergence towards a varifold solution, here a De~Giorgi solution.
    The argument is purely variational and does not rely on comparison principles. 
    The key novelty is an alternative proxy for the completely degenerate $L^2$ distance that is more robust than the one of Almgren-Taylor-Wang and Luckhaus-Sturzenhecker.
  
 \end{abstract}
	
	\maketitle
	
	\section{Introduction}
    
The (volume-preserving) Mean Curvature Flow (MCF) can be formally
interpreted as the $L^2$-gradient flow of the perimeter functional, possibly for fixed volume configurations.
Formally, this can be read off from the energy dissipation relation
\[
    \frac{\d}{\d t}\mathrm{Area}(\Sigma(t)) = -\int_{\Sigma(t)} V^2 \,\d S,
\]
which is valid for smooth solutions.
However, this gradient flow structure cannot be directly exploited for the purpose of well-posedness results. Indeed, the $L^2$
(geodesic) distance $\inf \int_{\tilde \Sigma(t)} \tilde V^2 \,\d S \, \d t$ is completely degenerate: any two configurations have zero $L^2$ distance \cite{Michor}.
Nevertheless, \textcolor{black}{in two independent landmark papers,} Almgren-Taylor-Wang \cite{ATW} and Luckhaus-Sturzenhecker \cite{LS} introduced an implicit time-discretization of the flow.
It is a natural time discretization of the Mean Curvature Flow and in fact inspired\footnote{See~\cite[Section 2.5.3, p.\ 33]{DeGiorgiSelectedPapers}.} De Giorgi's general definition of minimizing movements in abstract metric spaces~\cite{DeGiorgiMarinoTosques}, which led to a whole body of research on gradient flows, see e.g.~\cite{AGS}. De~Giorgi's basic idea is simple: in order to construct approximate solutions to a gradient flow, for a given time-step size $h>0$, starting from an initial condition $x_0$, successively solve variational problems of the form
\[
    x_n \in \arg\min_x E(x) + \frac{1}{2h} d^2(x,x_{n-1}).
\]
Ironically, due to the degeneracy of the metric in the case of the Mean Curvature Flow, the subsequent abstract theory~\cite{DeGiorgiMarinoTosques,AGS} does not apply to the framework of this original motivation. 
A central part of the work of Almgren-Taylor-Wang~\cite{ATW} and Luckhaus-Sturzenhecker~\cite{LS} is the introduction of a proxy for the squared $L^2$ distance given by $ d^2_{ATW,LS}(A,B) := \int_{A\triangle B} 2\dist(x,\partial A)\,\d x$.
However, it requires regularity theory of almost minimal surfaces even to prove a conditional convergence result for the resulting minimizing movements scheme, where one needs to assume the convergence of the perimeter in the limit of vanishing time-step size~\cite{LS} (see also \cite{Morini, Julin} for further properties of this scheme). 
This condition can only be shown under smoothness assumptions~\cite{Julin} or for mean convex singularities~\cite{DePhilippisLaux}. 
Another approach to construct a well-defined $L^2$-type distance due to Michor and Mumford~\cite{Michor} is to introduce a curvature regularization in the metric of the form $\int _{\partial A} (1+\alpha \kappa^2)V^2 \,\d S $. While this leads to a well-defined geodesic distance $d_{MM}^2(A,B) = \inf \int_0^1 \int _{\partial A_t} (1+\alpha \kappa^2)V^2 \,\d S \,\, \d t$, it is not useful for a minimizing movements scheme, as this generates higher-order terms in the evolution equation that cannot be controlled. 
Yet another regularization of the $L^2$ distance appeared in the Esedo\u{g}lu-Otto minimizing movements interpretation~\cite{EsedogluOtto} of the highly efficient Merriman-Bence-Osher scheme~\cite{MBO}, where $d_{EO,h}^2 (A,B) = c\sqrt{h} \int |G_{h}\ast (\xhi_A-\xhi_B)|^2\,\d x$ with $G_h$ the heat kernel at time $h$. This insight led to a series of (conditional) convergence results to distributional~\cite{LauxOtto, LauxSwartz}, Brakke~\cite{LauxOttoBrakke}, and De~Giorgi~\cite{LauxOttoDeGiorgi,LauxLelmiDeGiorgi,KraemerDeGiorgi} solutions. 
All these works assume the convergence of the energy in the limit of vanishing time-step size, which also in the case of the MBO scheme can only be shown for smooth flows~\cite{SwartzYip} or mean convex singularities~\cite{FuchsLaux}.
However, also for this approximation, an unconditional, global convergence result towards a varifold solution is still missing. 
In summary, while minimizing movements schemes are the most natural time discretizations of gradient flows and are well-studied in the case of the Mean Curvature Flow, prior to our work, no approximation for the degenerate $L^2$ distance has been found that allows for a global-in-time convergence proof towards any varifold solution. 

In the present work, we propose an alternative and more robust approximation of the $L^2$ geodesic distance. 
We regularize the $L^2$ metric with an additional Mullins-Sekerka-type nonlocal regularization and obtain a non-degenerate distance on sets. 
Loosely speaking, the metric tensor $\int_{\partial A} V^2 \, \d S$ expressed in terms of the indicator function $\xhi=\xhi_A$ is of the form $\int \big(\frac{\partial_t \chi}{|\nabla \chi|}\big)^2 |\nabla \chi|$, where the quotient in the square denotes the Radon-Nikod\'ym derivative. We approximate the (degenerate) geodesic distance by the proxy
\begin{align*}
d_{\alpha,h}^2(A,B)=\int \Big| \Big(|\rho_h \ast \nabla \chi_A| + \alpha^2(I-\Delta)\Big)^{-\frac12} (\chi_B-\chi_A) \Big|^2\, \d x
\end{align*}
with $\alpha, h >0$. Here $\rho_h$ is a standard mollifier on scale $h>0$ that regularizes the distributional gradient $\nabla \xhi_A$, and the additional operator $\alpha^2(I-\Delta)$ is a nonlocal Mullins-Sekerka-type regularization.
For $\alpha=h=0$, we formally recover the $L^2$ distance.
This regularization permits us to construct a new minimizing movements scheme, with De~Giorgi's interpolant, allowing to show that the limit, i.e., the so-called flat flow, is a De Giorgi varifold solution to the (volume preserving) Mean Curvature Flow in the sense of \cite{LH,Poiatti}.
The proxy for the distance in~\cite{ATW} and~\cite{LS}, even to be written down, requires some regularity from one of its arguments. 
In our case, we obtain a well-defined (albeit, still asymmetric) ``distance'' that does not require regularity, and most importantly, allows us to use the general theory of gradient flows~\cite{AGS} as long as $\alpha>0$. We can take the regularization $h$ to be equal to the time-step size, but take $\alpha \gg h$. 
In fact, it is convenient to take the disjoint limit, first $h\to0$, \textcolor{black}{to exploit the fundamental Mullins-Sekerka-type regularization}, then $\alpha\to0$.
Our main result, Theorem~\ref{thm1}, then shows that any limit of the minimizing movements scheme is a De~Giorgi varifold solution to (volume preserving) Mean Curvature Flow.

This is the first proof that a minimizing movements scheme guarantees unconditional existence of weak solutions (in the sense of De Giorgi varifold solutions) to (volume preserving) Mean Curvature Flow. Furthermore, we show that, in the case of volume preserving Mean Curvature Flow, we do not need to use any penalization to ensure the volume constraint as in the conditional result \cite{MSS}, since we can directly work with prescribed volume minimizers in the approximating scheme as is done in the recent \textcolor{black}{(also conditional) result}~\cite{Julin_global}. 

Furthermore, the present work  provides an alternative proof of the existence of such De Giorgi varifold solutions. The existence arguments in \cite{LH,Poiatti} rely on the sharp interface limit of the classical Allen-Cahn equation from Ilmanen \cite{Ilmanen} and, for the volume preserving case, its adaptation by \cite{Takasao} to guarantee volume conservation in the limit.
This new minimizing movements scheme and the convergence proof presented here have the potential to be extended to various other situations; for example, we envision capillary problems, coupling with advection \textcolor{black}{(cf. \cite{AbelsP} for the advective Mullins-Sekerka case)}, and possibly multiphase partition problems. \textcolor{black}{Finally, we note that here we present our theorem in dimensions $d=2,3$, for which we can crucially exploit the results by Sch\"atzle \cite{Schatzle} for chemical potentials in $H^1$. Nevertheless, we are confident that substituting the Laplacian for instance with a $p$-Laplacian, like an operator of the form $-\Div((1+\normmm{\nabla \cdot}^{2})^{\frac{p-2}2}\nabla \cdot)$, with $p>d/2$, and making the necessary changes in the metric to account for a nonlinear operator (e.g., defining it by duality), should allow to obtain similar results for all dimensions, exploiting again \cite{Schatzle} with chemical potentials in $W^{1,p}$. This will be objective of future investigations.}

	\section{Main results}

    \label{Sec::mainresults}
	First we give here the admissibility conditions for a couple given by evolving phase indicators and varifolds.
	\begin{definition}[Admissible couples of evolving phase indicators and varifolds]\label{admissible}
	Let $d=2,3$, and set $\Ts\in(0,\infty)$. Let $\xhi\in L^\infty_{w*}(0,\Ts; BV({\mathbb T^d};\{0,1\}))$. Consider a family $\mu=(\mu_t)_{t\in(0,\Ts)}$ of weakly* measurable in time oriented varifolds $\mu_t\in \mathcal M({\mathbb T^d}\times \mathbb S^{d-1})$, for all $t\in(0,\Ts)$. The couple $(\chi,\mu)$ is said to be admissible if the following properties are satisfied: 
	\begin{enumerate}
	
		\item\textit{(Compatibility with the phase indicator).} 
        The oriented varifold $\mu_t$ is compatible with the interface associated  with the phases modeled by $\chi$, in the sense that for almost every $t\in(0,\Ts)$, for every $\eta\in C^{1}({\mathbb T^d};\mathbb R^d)$ 
      
		\begin{align}
			&\int_{{\mathbb T^d}\times \mathbb{S}^{d-1}}p\cdot \eta(x)\,\d\mu_t(x,p)=\int_{\mathbb T^d}  \eta(x)\cdot \d{\nabla\chi(x,t)}.\label{PE}
		\end{align}

\item\textit{(Measurability in time of the energy).} The total mass measure associated with the oriented varifold $\mu_t$, that is
\begin{align}
	E[\mu_t]:=\normmm{\mu_t}_{\mathbb S^{d-1}}({\om}),
	\label{PL}	
\end{align}
is a measurable map from $t\in(0,\Ts)$ to $E[\mu_t]\in[0,\infty)$.
\end{enumerate}
\label{admissibility}
\end{definition}
 In order to give our definition of weak solutions to (volume preserving) Mean Curvature Flow, we introduce the following preliminary definitions: given $(\xhi,\mu)$ admissible as in Definition \ref{admissible}, we set, for almost any $t>0$,
\begin{align}\label{Schi}
        S_{\xhi_t}:=\begin{cases}C^1({\mathbb T^d};\R^d),&\quad\text{for MCF},\\\{B\in C^1({\mathbb T^d};\R^d):\ \int_{\mathbb T^d} \Div B\ \xhi(t)\dx =0\},&\quad\text{for volume preserving MCF},
        \end{cases}
    \end{align}
    where we note that in the volume preserving case the definition of $S_{\xhi_t}$ depends on the characteristic function $\xhi(t)$. We then introduce the Hilbert space
    \begin{align}
        \mathcal V_{\xhi_t}:=
            \overline{S_{\xhi_t}}^{L^2({\mathbb T^d};\d \normmm{\mu_t}_{\mathbb S^{d-1}})},\label{Vchi}
    \end{align}
    endowed with the $L^2({\mathbb T^d};\d \normmm{\mu_t}_{\mathbb S^{d-1}})$ inner product.
   We can now give our definition of weak solution, which is the same as \cite[Definition 1]{LH} for MCF, whereas it is slightly different  from the one for \textcolor{black}{volume preserving MCF} in \cite{Poiatti} (cf. Remark \ref{meancurvature}).
	\begin{definition}
		[Varifold solutions to (volume preserving) Mean Curvature Flow]
		Let $d=2,3$, consider $\Ts\in(0,\infty)$, and let $\mathbb T^d$ be the $d$-dimensional flat torus. Fix also $\chi_0\in BV({\mathbb T^d};\{0,1\})$ with $\int_{\mathbb T^d} \xhi_0\dx\in (0,\mathcal L^d({\mathbb T^d}))$. 
		A measurable map $\xhi:{\mathbb T^d}\times(0,\Ts)\to \{0,1\}$, together with a family $\mu=(\mu)_{t\in(0,\Ts)}$ of oriented varifolds $\mu_t\in \mathcal M (\om\times\mathbb{S}^{d-1})$, $t\in(0,\Ts)$, is called a \textit{varifold solution }for the Mean Curvature Flow, with time horizon $\Ts$ and initial datum $\xhi_0$, if 
		\begin{enumerate}
			\item\textit{(Admissibility and regularity).} $(\xhi,\mu)$ is an admissible couple in the sense of Definition \ref{admissibility}.
	
		\item\textit{(Square integrable velocity).} There exists a velocity $V\in L^2(0,\Ts;L^2({\mathbb T^d},\d\normmm{\mu_{(\cdot)}}_{\mathbb S^{d-1}}))$, such that 
		\begin{align}
		&\int_{\om}\xhi(\cdot,T')\zeta(\cdot, T')\dx-\int_{\om}\xhi_0\zeta(\cdot,0)\dx\nonumber\\&
        =\int_0^{T'}\int_{\mathbb T^d} \xhi(\cdot,t)\pt \zeta(\cdot,t)\dxdt-\int_0^{T'}\int_{\om}V\zeta\d\normmm{\mu_t}_{\mathbb S^{d-1}}\dt,
			\label{kinetic}
		\end{align} 
		for all $\zeta\in C^\infty_c({\mathbb T^d}\times[0,\Ts])$, for almost any $T'\in(0,\Ts)$; 
		
		\item \textit{(Sharp energy dissipation inequalities).} The following energy inequality holds: 
		\begin{align}
	        E[\mu_{T'}]+\onehalf\int_s^{T'}\int_{\mathbb T^d}\normmm{\ddt \xhi}^2\d \tmmu\dt +\onehalf\int_s^{T'} \normmm{\partial E(\mu_t)}^2_{\mathcal V_{\xhi_t}}\dt&\leq E[\mu_s],
			\label{advectiveMullins}	
		\end{align}
		 for almost any $T'\in(0,\Ts)$ and almost any $s\in[0,T')$, including $s=0$, where
		 $$
		 E[\mu_t]:=\begin{cases}
		 	\tmmu(\om),& t>0,\\
		 	E[\xhi_0],& t=0,
		 	\end{cases}
		 $$
		 with $E[\xhi_0]=\normmm{\nabla \chi_0}({\mathbb T^d})$,  
          \begin{align}
   \frac12 \int_s^{T'}\normmm{\partial E(\mu_t)}^2_{\mathcal V_{\xhi_t}} :=\sup_{B\in C^1_c((s,T');\mathcal S_{\xhi_{(\cdot)}})}\left(\int_s^{T'}\delta \mu_t(B(t))\dt-\frac12\int_s^{T'}\norm{B(t)}_{\mathcal V_{\xhi_{t}}}^2\dt\right),\label{slope}
\end{align}
\textcolor{black}{and we identified, thanks to \eqref{kinetic}, $\ddt \xhi$ with $V$.}
		\end{enumerate}	 
		\label{weaksol}   
	\end{definition}
    \begin{remark}[Existence of a generalized mean curvature vector]\label{meancurvature}
        Compared to the notions in \cite{LH,Poiatti}, here we adopted a \textcolor{black}{more formal} gradient flow perspective in the energy inequality \eqref{advectiveMullins}. This allows us to obtain the existence of a generalized mean curvature vector directly as a consequence of \eqref{advectiveMullins}. Indeed, in the case of MCF, since $\onehalf\int_0^{T^*}  \normmm{\partial E(\mu_t)}^2_{\mathcal V_{\xhi_t}}\dt\leq C<\infty$, from \eqref{slope} we immediately infer (after uniquely extending by density the first variation $\delta\mu_t$ on $L^2(0,\Ts;L^2({\mathbb T^d}; \d\tmmu;\R^d))$) that there exists a vector field $\mathbf H_{\normmm{\mu}_{\mathbb S^{d-1}}}\in L^2(0,\Ts;L^2({\mathbb T^d}; \d\normmm{\mu_{(\cdot)}}_{\sf};\R^d))$ satisfying, for almost any $t\in(0,\Ts)$, 

	\begin{align}
		\int_0^\Ts\delta\mu_t(B)=-\int_0^\Ts\int_\om \mathbf H_{\tmmu}\cdot B\,\d\tmmu,
		\label{PH}
	\end{align}
    where $\delta\mu_t$ is the first variation of $\mu_t$ in the direction of a vector field $B\in C_c^1((0,\Ts)\times\om; \mathbb R^d)$. Also, it holds 
$$
 \frac12 \int_s^{T'}\normmm{\partial E(\mu_t)}^2_{\mathcal V_{\xhi_t}}\dt=\frac12 \int_s^{T'}\int_{{\mathbb T^d}} \normmm{\mathbf H_{\normmm{\mu_t}_{\mathbb S^{d-1}}}}^2\d\tmmu\dt,\quad \forall s,T'\in(0,\Ts),
$$
so that the energy inequality is equivalent to the one stated in \cite[Definition 1]{LH}. On the other hand, in the volume preserving case, the definition of $\mathcal V_{\xhi_t}$ and \eqref{slope} gives the existence of a vector field $\mathbf H^0_{\normmm{\mu}_{\mathbb S^{d-1}}}\in L^2(0,\Ts;\mathcal V_{\xhi(\cdot)})$, satisfying \eqref{PH} for any $B\in C^1_c((0,\Ts);S_{\xhi_{(\cdot)}})$ and for almost any $t\in(0,\Ts)$, and such that 
$$
\frac12 \int_s^{T'} \normmm{\partial E(\mu_t)}^2_{\mathcal V_{\xhi_t}}\dt=\frac12 \int_s^{T'}\int_{{\mathbb T^d}} \normmm{\mathbf H^0_{\normmm{\mu_t}_{\mathbb S^{d-1}}}}^2\d{\normmm{\mu_t}_{\sf}}\dt,\quad\forall s,T'\in(0,\Ts),
$$
which is in general different from the corresponding term in the energy inequality in \cite[(3.3d)]{Poiatti}. The main difference is the fact that $\mathbf H^0_{\normmm{\mu}_{\mathbb S^{d-1}}}$ enjoys a further interesting generalized zero integral average property, namely
\begin{align}
   \int_{{\mathbb T^d}\times \mathbb S^{d-1}} \mathbf H^0_{\normmm{\mu_t}_{\mathbb S^{d-1}}}\cdot p\ \d\mu_t=0,\quad\text{ for a.a. t}\in(0,\Ts)\label{zero}.
\end{align}
Indeed, since $\mathbf H^0_{\normmm{\mu_t}_{\mathbb S^{d-1}}}\in \mathcal V_{\xhi_t}$, there exists a sequence $\{B_k\}_k\subset \mathcal S_{\xhi_t}$ such that $B_k\to \mathbf H^0_{\normmm{\mu_t}_{\mathbb S^{d-1}}}$ in $L^2({\mathbb T^d};\d \normmm{\mu_t}_{\sf})$, and thus we can write, by the Gau{\ss} Theorem and the compatibility condition \eqref{PE},
\begin{align*}
   0= \int_{\mathbb T^d} \Div B_k\ \xhi_t\dx =-\int_{\mathbb T^d} B_k\cdot \bn_t\ \d\normmm{\nabla \xhi_t}=-\int_{{\mathbb T^d}\times \mathbb S^{d-1}}B_k\cdot p\ \d\mu_t,
\end{align*}
giving \eqref{zero} by passing to the limit as $k\to\infty$. To conclude, thanks to Lemma \ref{curvature} in the appendix, one can find a vector field $\mathbf H_{\normmm{\mu_t}_{\mathbb S^{d-1}}}\in L^2(0,\Ts;L^2({\mathbb T^d};\d\normmm{\mu_{(\cdot)}}_{\mathbb S^{d-1}}))$ such that \eqref{PH} actually holds for any $B\in C_c^1((0,\infty)\times {\mathbb T^d};\R^d)$. Namely, Lemma \ref{curvature}, with $m_0:=\int_{\mathbb T^d} \xhi(t)\dx=\int_{\mathbb T^d} \xhi_0\dx \in (0,\mathcal L^d({\mathbb T^d}))$ for almost any $t\geq0$, gives the existence of $\lambda:(0,\Ts)\to \R$ such that
\begin{align}
    \nonumber\normmm{\lambda(t)}&\leq C(1+\normmm{\nabla\chi(t)}(\om))\left(\normmm{\mu_t}_{\mathbb S^{d-1}}(\om)+\sqrt{\normmm{\mu_t}_{\mathbb S^{d-1}}(\om)}\norm{\mathbf H^0_{\normmm{\mu_t}_{\mathbb S^{d-1}}}}_{L^2({\mathbb T^d};\d\normmm{\mu_t}_{\mathbb S^{d-1}})}\right)\\&\leq C\left(1+\norm{\mathbf H^0_{\normmm{\mu_t}_{\mathbb S^{d-1}}}}_{L^2({\mathbb T^d};\d\normmm{\mu_t
    }_{\mathbb S^{d-1}})}\right),\label{lambdat}
\end{align}
exploiting the energy inequality \eqref{advectiveMullins}, together with the compatibility $\normmm{\nabla \xhi(t)}\leq \normmm{\mu_t}_{\sf}$ for almost any $t>0$. This gives $\lambda\in L^2(0,\Ts)$. Then, disintegrating the varifold as $\mu_t=\normmm{\mu_t}_{\mathbb S^{d-1}}\otimes \left(\pi_x\right)_{x\in {\mathbb T^d}}$, with $\pi_x$ probability measure, \eqref{PH} holds if we set
\begin{align}
\label{full}
\mathbf H_{\normmm{\mu_t}_{\mathbb S^{d-1}}}(x):=\mathbf H^0_{\normmm{\mu_t}_{\mathbb S^{d-1}}}(x)+\lambda(t)\int_{\mathbb S^{d-1}}p\d\pi_x,\quad\text{for }\normmm{\mu_t}_{\mathbb S^{d-1}}\text{-a.a. }x\in {\mathbb T^d},
\end{align}
just using \eqref{totvar} and recalling the compatibility \eqref{PE}. 
\end{remark}
The first existence proofs of such weak solutions to (volume preserving) Mean Curvature Flow can be found in \cite{LH,Poiatti}, which are strongly based on the sharp interface limit of the classical Allen-Cahn equation studied in the fundamental work by Ilmanen \cite{Ilmanen} and, for the volume preserving case, its \textcolor{black}{generalization by Takasao} \cite{Takasao} to ensure the volume conservation in the limit. Here we propose an alternative self contained proof that is valid, up to small adaptations, for showing the existence of both the Mean Curvature Flow and the volume preserving Mean Curvature Flow. The argument is for the first time entirely based on a minimizing movements approach, using a suitable De Giorgi interpolant. Therefore no maximum principles are used, as it is the case for the proofs in \cite{LH,Poiatti}, where a crucial maximum principle is exploited in the study the sign of the so-called discrepancy measure, necessary if one adopts a diffuse-to-sharp interface approach.
We then state our main existence result.
\\ \ \\
\begin{theorem}[Varifold solutions to Mean Curvature Flow as limits of minimizing movements]
	Let $d \in \{2,3\}$, $\mathbb T^d$ be the $d$-dimensional flat torus. Let $\xhi_0 \in BV({\mathbb T^d};\{0,1\})$.
	Then, there exists a suitable minimizing movements scheme such that its limit (flat flow) is a varifold solution to the Mean Curvature Flow and to the volume preserving Mean Curvature Flow,  globally defined on $(0,\infty)$, with
	initial datum $\chi_0$, in the sense of Definition \ref{weaksol} for any $\Ts>0$.

	\label{thm1}
\end{theorem}  
\ \\ \ 
\begin{remark}
\label{diagonal}
    As it will be seen in the proof, the minimizing movements scheme admittedly relies on two parameters, say the step size $h>0$ and the parameter $\alpha>0$ of the Mullins-Sekerka-type approximation. The order of the limits is first as $h\to0$, so that we can exploit the regularization properties of Mullins-Sekerka, namely the fundamental results by \cite{Schatzle}, and finally as $\alpha\to 0$. Of course, with suitable diagonalization procedures, one can extract a subsequence of parameters for which the minimizing movements solution converges to a De Giorgi varifold solution in a single passage to the limit. This shows that there exists a flat flow solution from a suitable minimizing movements scheme which is also a De Giorgi varifold solution.  
\end{remark}
\begin{remark}[On the generalized mean curvature vector]
Remark \ref{meancurvature} holds for any $\Ts>0$. Namely, in the case of Mean Curvature Flow, we directly obtain the existence of a generalized mean curvature vector such that $\mathbf H_{\normmm{\mu}_{\mathbb S^{d-1}}}\in L^2(0,\infty;L^2({\mathbb T^d}; \d\normmm{\mu_{(\cdot)}}_{\sf};\R^d))$ and \eqref{PH} holds with $\Ts=\infty$. On the other hand, for the volume preserving Mean Curvature Flow we deduce that $\mathbf H_{\normmm{\mu}_{\mathbb S^{d-1}}}^0\in L^2(0,\infty;L^2({\mathbb T^d}; \d\tmmu;\R^d))$ and, recalling \eqref{lambdat}-\eqref{full}, the generalized mean curvature vector, satisfying \eqref{PH} with $\Ts=\infty$, is such that $\mathbf H_{\normmm{\mu}_{\mathbb S^{d-1}}}\in L^2_{loc}((0,\infty);L^2({\mathbb T^d}; \d\tmmu;\R^d))$. 
\end{remark}
  \begin{remark}[On the rectifiability of the varifolds]
   Since, from Remark \ref{meancurvature}, there exists for both types of flows a generalized mean curvature vector $\mathbf{H}_{\normmm{\mu}_{\mathbb S^{d-1}}} \in L^2_{loc}((0,\infty);L^2({\mathbb T^d};d\normmm{\mu_{(\cdot)}}_{\mathbb S^{d-1}};\R^d))$, it is immediate to infer that  $\int_{\mathbb T^d} \normmm{\bH_{\normmm{\mu_t}_{\mathbb S^{d-1}}} }^2\d\normmm{\mu_t}_{\mathbb S^{d-1}}<+\infty$ for almost any $t>0$. Therefore, Allard's rectifiability criterion \cite{Allard} allows to conclude that the unoriented varifold $\widehat\mu_t$, corresponding to $\mu_t$  (identifying antipodal points $\pm p$ on $\mathbb S^{d-1}$), is $(d-1)$-rectifiable for almost any $t\in(0,\infty)$. On the other hand, in general we do not know anything about the integer rectifiability of the varifold, which is apparently nontrivial to be obtained passing through a minimizing movements scheme. So far, integrality of the varifolds can be obtained if we construct De Giorgi solutions through a sharp interface limit procedure (cf. \cite{LH,Poiatti}).  
\end{remark}
  \begin{remark}[Weak-strong uniqueness and consistency principles for the flat flow]
      Thanks to Remark \ref{meancurvature}, the notions of weak solutions for (volume preserving) MCF  can be reduced to the ones in \cite{ATW,Poiatti}. Therefore a weak-strong uniqueness principle holds, see \cite[Theorem 2]{LH} and \cite[Theorem 3.12]{Poiatti}. Crucially, both the results do not rely on the integer rectifiability of the varifold, thus they are perfectly applicable also in this case, showing a fundamental consistency principle of this notion of solution. This has the fundamental consequence that, similarly to \cite{ATW,Julin}, if there exists a smooth solution on $[0,\Ts]$ for some $\Ts>0$, then by the weak-strong uniqueness principle above \textit{any} flat flow solution found in Theorem \ref{thm1} is indeed smooth. In this special case, recalling Remark \ref{diagonal}, we can even find a unique sequence, choosing, e.g., $\alpha=\frac1n$ and a suitable $h=h(n)$, so that the corresponding flat flow solution is smooth.
  \end{remark}      
     
\section{Proof of Theorem \ref{thm1}} 
We aim at showing that there exists a minimizing movements approach, with a substantially different scheme from the one  in \cite{LS, ATW} for MCF and in \cite{MSS} for volume preserving MCF, such that its corresponding flat flow solution (i.e., the limit as the approximating parameters vanish) is a De Giorgi varifold solution to (volume preserving) Mean Curvature Flow.
\label{proofthm1}
	\subsection{First discretized step}\label{firstminmovement}
	Set $h>0$ to be the time-step size. For $s\geq0$, we introduce 
	$$T_h(s):=s-(n-1)h,\quad s\in((n-1)h,nh],$$
	where $n=\ceil{\tfrac sh}$, so that $T_h((mh)^-):=\lim_{s\to (mh)^-}T_h(s)=h$, for any $m\in\N$, and $T_h(t)=t$ for $t\in(0,h]$. Let us also consider the initial datum $\xhi_0\in BV({\mathbb T^d};\{0,1\})$ and define the set of competitors for the minimization problems, which depends on the fact that we are considering MCF or volume preserving MCF: 
\begin{align*}
    \mathcal M:=\begin{cases}
        BV({\mathbb T^d};\{0,1\}),&\quad \text{for MCF,}\\
        \{\xhi\in BV({\mathbb T^d};\{0,1\}):\ \int_{\mathbb T^d} \xhi \dx=\int_{\mathbb T^d} \xhi_0\dx\},&\quad\text{for volume preserving MCF}.
    \end{cases}
\end{align*}
	\subsubsection{minimizing movements step}
    Let us consider a standard symmetric mollifier such that 
\begin{align*}
    \rho\in C^\infty_c(B_1(0)),\quad \int_{\mathbb T^d} \rho(x)\dx=1,\quad \rho\geq 0,\quad \rho(x)=\rho(\normmm{x}),
\end{align*}
    and define 
    \begin{align*}
        \rho_h(x):=h^{-d}\rho\left(\frac{x}h\right).
    \end{align*}
    Then we consider the following convolution, given $\xhi\in BV({\mathbb T^d};\{0,1\})$:
    $$
    \rho_h\ast \xhi(x):=\int_{\mathbb T^d} \rho_h(x-y)\xhi(y)\d y,\quad \rho_h\ast\xhi\in C^\infty({\mathbb T^d}),
    $$
so that (cf. \cite[Proposition 12.20]{Maggi}),
\begin{align*}
    \rho_h\ast \nabla \xhi=\nabla (\rho_h\ast \xhi)\in C^\infty({\mathbb T^d}).
\end{align*}
Since we are working on the torus $\mathbb T^d$, we can consider the Laplace operator $-\Delta$, so that the operator $I-\Delta$, with $I$ identity operator, is invertible from $L^2({\mathbb T^d})$ onto $H^2({\mathbb T^d})$. Let us now fix $\alpha>0$ and consider $\xhi\in BV({\mathbb T^d};\{0,1\})$. Then the operator 
\begin{align*}
   \mathcal A:= \normmm{\rho_h\ast \nabla \xhi}+\alpha^2(I-\Delta)
\end{align*}
is invertible from $L^2({\mathbb T^d})$ onto $H^2({\mathbb T^d})$. Indeed, it is immediate to see that the bilinear form $a:H^1({\mathbb T^d})\times H^1({\mathbb T^d})\to H^1({\mathbb T^d})'$ defined by $a(u,v)=\langle\mathcal A u,v\rangle$ is coercive, since
$$
a(u,u)\geq \alpha^2\left(\norm{u}^2_{L^2({\mathbb T^d})}+\norm{\nabla u}^2_{\mathbf L^2({\mathbb T^d})}\right)+\int_{\mathbb T^d} \normmm{\rho_h\ast \nabla \chi}\normmm{u}^2\dx\geq \alpha^2\norm{u}_{H^1({\mathbb T^d})}^2,
$$
and thus by the Lax-Milgram Lemma the operator $\mathcal A$ is invertible from $H^1({\mathbb T^d})'$ to $H^1({\mathbb T^d})$. Then, since the lower order perturbation $\normmm{\rho_h\ast \nabla \chi}$ is Lipschitz continuous, we immediately infer by elliptic regularity that the unbounded operator $\mathcal A: H^2({\mathbb T^d})\hookrightarrow\hookrightarrow L^2({\mathbb T^d})\to L^2({\mathbb T^d})$ is invertible on $L^2({\mathbb T^d})$ with compact inverse.  As a consequence, by the spectral theorem for strictly positive operators we can define the real powers of $\mathcal A$, namely $\mathcal A^s$ for any $s\in\R$. Therefore, the operator $\mathcal A^{-\frac12}=\left(\normmm{\rho_h\ast{\nabla \xhi}}+\alpha^2(I-\Delta)\right)^{-\frac12}: H^1({\mathbb T^d})'\to H^1({\mathbb T^d})$ is well defined and continuous. 

We are now ready to introduce the following minimization problem, for $h>0$ and $\alpha>0$ ,
	\begin{align}
		&\chi^h_{0}=\chi_0,\\
		&\chi_{1}^{h}\in\arg\min_{\chi\in \mathcal M}\left(E[\chi]+\frac 1{2h}\normh{\left(\normmm{\rho_h\ast{\nabla \xhi_0^h}}+\alpha^2(I-\Delta)\right)^{-\frac12}(\chi-\chi_{0}^{h})}^2\right),\label{M1}
	\end{align}
	with
	$$
	E[\xhi]=\normmm{\nabla \xhi}({\mathbb T^d}).
	$$
	We further need the De Giorgi interpolation, as follows:
	\begin{align}
		& \xhi^{h}_{DG}(0):=\xhi_{0}^h,\\&
		\xhi_{DG}^h(t)\in\arg\min_{\chi\in \mathcal M}\left(E[\chi]+\frac 1{2t}\normh{\left(\normmm{\rho_h\ast{\nabla \xhi_0^h}}+\alpha^2(I-\Delta)\right)^{-\frac12}(\chi- \xhi^{h}_{0})}^2\right),
		\text{ for }t\in (0,h].\label{M2}
	\end{align}
	For any fixed $t$, the problem clearly admits at least one solution by the Direct Method of the Calculus of Variations. The solution can be chosen to be strongly measurable with values in $L^p({\mathbb T^d})$, for any $p>1$. This can be shown in a similar way as in \cite[Appendix B]{AbelsP}.

Let us now define the function 
 \begin{align}
 	f(t):=E[\xhi_{DG}^h(t)]+\frac{1}{2t}\normh{\left(\normmm{\rho_h\ast{\nabla \xhi_0^h}}+\alpha^2(I-\Delta)\right)^{-\frac12}(\chi_{DG}^h(t)- \xhi^{h}_{0})}^2,
 	\label{ftt}
 \end{align}
 where $\xhi_{DG}^h(t)$ is defined in \eqref{M2}.
 By standard arguments from~\cite{AGS} (see, e.g., \cite{SH,AbelsP}) it can be shown that $f(t)$ is locally Lipschitz continuous on $(0,h]$ and, for almost any $t\in(0,h]$,
 \begin{align}
 \ddt f(t)&=-\frac{1}{2t^2}\normh{\left(\normmm{\rho_h\ast{\nabla \xhi_0^h}}+\alpha^2(I-\Delta)\right)^{-\frac12}(\chi_{DG}^h(t)- \xhi^{h}_{0})}^2.\label{derivative}
 \end{align}

\subsection{Iterative scheme at step $n>1$}
Let us assume to know $\xhi_{DG}^h(t)$ up to step $n-1$, i.e., we assume $\xhi_{DG}^h(t)$ to be defined on $[0,(n-1)h]$, $n\in\N$, $n>1$. Let also $\xhi^h_{m-1}$ be given for any $0\leq m\leq n$. We can now define the minimizing step $n$. 
\subsubsection{Minimizing movements scheme}
We set the following minimizing sequence:
\begin{align}
	&\chi_{n}^{h}\in\arg\min_{\chi\in \mathcal M}\left(E[\chi]+\frac 1{2h}\normh{\left(\normmm{\rho_h\ast{\nabla \xhi_{n-1}^h}}+\alpha^2(I-\Delta)\right)^{-\frac12}(\chi- \xhi^{h}_{n-1})}^2\right),\label{M1b}
\end{align}

Then, we define the De Giorgi interpolant as:
\begin{align}
	& \xhi^h_{DG}((n-1)h):=\xhi_{n-1}^h,\\&
	\xhi_{DG}^h(t)\in\arg\min_{\chi\in \mathcal M}\left(E[\chi]+\frac 1{2(t-(n-1)h)}\normh{\left(\normmm{\rho_h\ast{\nabla \xhi_{n-1}^h}}+\alpha^2(I-\Delta)\right)^{-\frac12}(\chi- \xhi^{h}_{n-1})}^2\right),\nonumber\\&
	\text{ for }t\in ((n-1)h,nh].\label{M2b}
\end{align}
As already observed, for any fixed $t$, the problem admits at least one solution by the Direct Method of the Calculus of Variations. Observe that \eqref{derivative} is still valid up to substituting $\xhi_0$ with $\xhi_{m}^h$, for all $1\leq m\leq n-1$. We then define, for any $1\leq m\leq n$, 
\begin{align}
&\label{uh1}u_h^m:=-\left(\normmm{\rho_h\ast{\nabla \xhi_{m-1}^h}}+\alpha^2(I-\Delta)\right)^{-1}\frac1{h}\left(\xhi_m^{h}-\xhi_{m-1}^h\right),\\&
w_h(t):=-\left(\normmm{\rho_h\ast{\nabla \xhi_{m-1}^h}}+\alpha^2(I-\Delta)\right)^{-1}\frac{1}{\Th{t}}\left(\chi_{DG}^{h}(t)-\xhi_{m-1}^h\right),\quad t\in((m-1)h,mh),\label{wh}
\end{align}
which are well defined quantities thanks to the discussion above on the invertibility of the operator $\mathcal A$.

We also introduce the piecewise constant interpolants as 
\begin{align*}
	\xhi^h(t):=\xhi_{m-1}^h,\quad  u_h(t):=u_h^{m-1}\quad t\in[(m-1)h,mh),\quad 1\leq m\leq n,
\end{align*}
together with the affine interpolant

\begin{align*}
\widehat\xhi^h(t):=\frac{t-(m-1)h}{h}\xhi_m^h+\frac{mh-t}{h}\chi^h_{m-1},\quad t\in[(m-1)h,mh], \quad 1\leq m\leq n.
\end{align*}
Notice that, while $\xhi^h(t)$ and $\chi_{DG}^{h}(t)$ take values in $\{0,1\}$, the piecewise affine interpolant $\widehat \chi^h(t)$ takes values in the entire interval $[0,1]$.
\subsection{Global energy estimates}
We now iterate the procedure of the previous section for any $n\in\mathbb N$. Recall that on any interval of the form $((n-1)h,nh)$, $n\in\N$, we have that \eqref{derivative}  holds, namely we have 
\begin{align}
	&	 E[\xhi_{n}^{h}]+\frac1{2h}\normh{\left(\normmm{\rho_h\ast{\nabla \xhi^{h}_{n-1}}}+\alpha^2(I-\Delta)\right)^{-\frac12}(\chi_n^h- \xhi^{h}_{n-1})}^2\nonumber\\&\label{energy1a}+\int_{(n-1)h}^{nh}\frac{1}{2{T_h(\tau)}^2}\normh{\left(\normmm{\rho_h\ast{\nabla\xhi^{h}_{n-1}}}+\alpha^2(I-\Delta)\right)^{-\frac12}(\chi_{DG}^{h}(\tau)-\xhi^{h}_{n-1})}^2\d \tau\\& \nonumber\leq E[\xhi_{n-1}^h].
\end{align}
We then introduce the function
\begin{align*}
	g_h(t):=t-\floor{\frac{t}h}h,
\end{align*}
so that, from \eqref{energy1a}, which is valid for any $n\in\mathbb N$, we immediately infer, by a telescoping argument, that, for any $n,m\in \mathbb N$, $m<n$, 
\begin{align}
	\nonumber	&E[\xhi_{n}^{h}]+\frac1{2}\int_{mh}^{nh}\normh{	\left(\normmm{\rho_h\ast{\nabla \xhi^h(\tau)}}+\alpha^2(I-\Delta)\right)^{-\frac12}\partial_t{\widehat\xhi}_h(\tau)}^2\d\tau\nonumber\\&+\frac12\int_{mh}^{nh}\frac{1}{{g_h(\tau)}^2}\normh{\left(\normmm{\rho_h\ast{\nabla \xhi^h(\tau)}}+\alpha^2(I-\Delta)\right)^{-\frac12}(\chi_{DG}^{h}(\tau)-\xhi^h(\tau))}^2\d \tau\nonumber\\& \leq E[\xhi_{m}^h].
	\label{energy1ab1}
\end{align}
Observe that the same inequality can be rewritten as 
\begin{align}
	\nonumber	&E[\xhi_{n}^{h}]+\frac1{2}\int_{mh}^{nh}\normh{	\left(\normmm{\rho_h\ast{\nabla \xhi^h(\tau)}}+\alpha^2(I-\Delta)\right)^{\frac12}u_h(\tau)}^2\d\tau\nonumber\nonumber\\&+\int_{mh}^{nh}\normh{\left(\normmm{\rho_h\ast{\nabla \xhi^h(\tau)}}+\alpha^2(I-\Delta)\right)^{\frac12}w_h(\tau)}^2\d \tau\nonumber\\& \leq E[\xhi_{m}^h],
	\label{energy1ab}
\end{align}
where it holds
\begin{align}
    \nonumber\normh{\left(\normmm{\rho_h\ast{\nabla \xhi^h(\tau)}}+\alpha^2(I-\Delta)\right)^{\frac12}f}^2&=\int_{\mathbb T^d} f^2\normmm{\rho_h\ast \nabla\xhi^h(\tau)}\dx +\alpha^2(\normh{f}^2+\normh{\nabla f}^2)
    \\&
    \geq \alpha^2\norm{f}_{H^1({\mathbb T^d})}^2.\label{contr1}
\end{align}
By following, for instance, \cite[Proof of Theorem 1, Step 1]{SH} (see also \cite{AGS}), we can easily verify that
\begin{align}
t\mapsto E[\xhi_{DG}^h(t)]\quad\text{ is monotone nonincreasing on } (0,\infty).
    \label{monotonicity}
\end{align}
Also, from \eqref{energy1ab} it is immediate to infer that
\begin{align}
t\mapsto E[\xhi^h(t)]\quad\text{ is monotone nonincreasing on }(0,\infty).
    \label{monotonicity2}
\end{align}

Let us now fix $T^*>0$, and $0<s<\kappa<\tau<T<T^*$. We can assume w.l.o.g. $h<\min\{T-\tau,\kappa-s\}$, so that we can set $n_0:=\floor{\frac{T}h}$ and $m_0:=\ceil{\frac sh}$, entailing
$$
0<s\leq m_0h<\kappa<\tau<n_0h\leq T,\quad T\in[n_0h,(n_0+1)h),\quad s\in((m_0-1)h,m_0h].
$$
Then we find the two inequalities 
\begin{align}
	\nonumber	&E[\xhi_{DG}^h({\tau)}]+\frac1{2}\int_{0}^{\tau}\normh{	\left(\normmm{\rho_h\ast{\nabla \xhi^h(t)}}+\alpha^2(I-\Delta)\right)^{-\frac12}\partial_t{\widehat\xhi}^h(t)}^2\dt\\&\nonumber+\int_{0}^{\tau}\frac{1}{2{g_h(t)}^2}\normh{\left(\normmm{\rho_h\ast{\nabla \xhi^h(t)}}+\alpha^2(I-\Delta)\right)^{-\frac12}(\chi_{DG}^{h}(t)-\xhi^h(t))}^2\d t\\& \leq E[\xhi_0]
	\label{energy1abc1}
\end{align}
as well as
\begin{align}
	\nonumber	&E[\chi_{DG}^{h}(T)]+\frac1{2}\int_{\kappa}^{\tau}\normh{	\left(\normmm{\rho_h\ast{\nabla \xhi^h(t)}}+\alpha^2(I-\Delta)\right)^{-\frac12}\partial_t{\widehat\xhi}^h(t)}^2\dt\\&\nonumber+\int_{\kappa}^{\tau}\frac{1}{2{g_h(t)}^2}\normh{\left(\normmm{\rho_h\ast{\nabla \xhi^h(t)}}+\alpha^2(I-\Delta)\right)^{-\frac12}(\chi_{DG}^{h}(t)-\xhi^h(t))}^2\d t\\& \leq E[\chi_{DG}^{h}(s)]\leq E[\xhi_0].
	\label{energy1abc}
\end{align}
Additionally, notice that by \eqref{monotonicity} we infer, for any $t\geq0$,
\begin{align}
E[\xhi_{DG}^h(t)]\leq E[\xhi^h(t)]\leq E[\chi_{DG}^{h}(s)], \quad \forall s <t-h,
    \label{fundamental}
\end{align}
where we have set $\xhi_{DG}^h(s)=\xhi_0$ for $s<0$.

Also, we observe here that, by Lemma \ref{Lemmaequiv}, we have a control on the norm
\begin{align}
\nonumber&\norm{f}_{(L^\infty({\mathbb T^d})\cap H^1({\mathbb T^d}))'}\leq \sqrt 3\sqrt{\max\{\normmm{\nabla \xhi^h(t)}({\mathbb T^d}),\alpha^2\}}\normh{\left(\normmm{\rho_h\ast {\nabla\xhi^h(t)}}+\alpha^2(I-\Delta)\right)^{-\frac12}f}\\&
\leq \sqrt 3\sqrt{\max\{\normmm{\nabla \xhi_0}({\mathbb T^d}),\alpha^2\}}\normh{\left(\normmm{\rho_h\ast {\nabla\xhi^h(t)}}+\alpha^2(I-\Delta)\right)^{-\frac12}f},
\label{equiv}
\end{align}
where we used \eqref{energy1abc} and \eqref{fundamental}.
Following \cite{SH}, we introduce the varifolds associated to the varifold lift of the De Giorgi interpolant $\xhi^{h}_{DG}(t)$ in ${\mathbb T^d}$ as 
\begin{align*}
	\widetilde\mu_t^{h}:=\normmm{\nabla\xhi^{h}_{DG}(t)}\otimes \left(\delta_{\frac{\nabla \xhi^{h}_{DG}(t)}{\normmm{\nabla\xhi^{h}_{DG}(t)}}(x)}\right)_{x\in {\mathbb T^d}}\in \mathcal M({{\mathbb T^d}}\times \mathbb S^{d-1}),
\end{align*}
 so that for any $t\in(0,T^*)$ it holds $E[\xhi_{DG}^h(t)]=\normmm{\mu_t^h}_{\mathbb S^{d-1}}({\mathbb T^d})$.

 Analogously, we introduce the varifold lift of $\xhi^{h}(t)$ in ${\mathbb T^d}$ as 
\begin{align*}
	\mu_t^h:=\normmm{\nabla\xhi^{h}(t)}\otimes \left(\delta_{\frac{\nabla \xhi^{h}(t)}{\normmm{\nabla\xhi^{h}(t)}}(x)}\right)_{x\in {\mathbb T^d}}\in \mathcal M({{\mathbb T^d}}\times \mathbb S^{d-1}),
\end{align*}

so that for any $t\in(0,T^*)$ it holds $E[\xhi^h(t)]=\normmm{\mu_t^h}_{\mathbb S^{d-1}}({\mathbb T^d})$.

\subsection{Uniform estimates in $h$}

 We thus obtain all the desired uniform estimates from \eqref{energy1abc1}:
 \begin{align}
 	\esssup_{t \geq0}E[\xhi_{DG}^h(t)]\leq E[\xhi_0],
 	\label{dissiprel}
 \end{align}
 and also, recalling the controls \eqref{contr1} and \eqref{equiv},
 \begin{align}
 	&	\alpha\norm{u_h}_{L^2(0,T;H^1({\mathbb T^d}))}+\alpha\norm{w_h}_{L^2(0,T;H^1({\mathbb T^d}))}+\norm{\partial_t{\widehat\xhi}^h}_{L^2(0,T;(L^\infty(\om)\cap H^1({\mathbb T^d}))')}\leq C(E[\xhi_0]),\label{regvbis}
 \end{align}
 for any $T<\infty$. Additionally, we also have 
\begin{align}
    \int_0^T\int_{\mathbb T^d} u_h(t)^2\normmm{\rho_h\ast{\nabla \xhi^h(t)}}\dx\dt+\int_0^T\int_{\mathbb T^d} w_h(t)^2\normmm{\rho_h\ast{\nabla \xhi^h(t)}}\dx\dt\leq E[\xhi_0].\label{hutchinson}
\end{align}
Note that $u_h$ satisfies, for any $\zeta\in C_c^\infty([0,\infty)\times \om)$,
\begin{align}
&\nonumber\int_0^{\infty}\int_{\mathbb T^d}\partial_t\widehat\xhi^h(t)\zeta(t)\dx\dt\label{timederiv}\\&\nonumber=
-\int_0^\infty\int_{\mathbb T^d} u_h(t)\zeta(t)\normmm{\rho_h\ast{\nabla \xhi^h(t)}}\dx\dt-\alpha^2\int_0^\infty\int_{\mathbb T^d} u_h(t)\zeta(t)\dx\dt\\&\quad  -\alpha^2\int_0^{\infty}\int_{\mathbb T^d} \nabla u_h(t)\cdot \nabla\zeta(t)\dx\dt.
\end{align}

Also, for all $t\in(0,T^*)$, by the minimizing property \eqref{M2} of $\xhi^{h}_{DG}(t)$, Allard’s first
variation formula \cite{Allard}, and \cite[Lemma 10]{SH}, it follows that $\xhi_{DG}^h(t)$ also satisfies the approximate Gibbs-Thomson relation:
\begin{align}
	\int_{{\mathbb T^d}\times \mathbb S^{d-1}}(Id-s\otimes s):\nabla B(x)\, \d \widetilde\mu_t^h(x,s)=\int_{\mathbb T^d} \xhi_{DG}^h(t)\text{div}(w_h(\cdot,t)B)\,\d x.\label{allard}
\end{align}
for all $t\in(0,T^*)$ and all $B\in S_{\xhi^{h}_{DG}(t)}$. Observe that, recalling the definition of $S_{\xhi^{h}_{DG}(t)}$, this space is simply $C^1({\mathbb T^d};\R^d)$ in the MCF case, whereas it actually depends on ${\xhi^{h}_{DG}(t)}$ for the volume preserving MCF case.   Also, the same relation is satisfied by $\xhi^h(t)$, just recalling the definition in the minimizing movements scheme:
\begin{align}
	\int_{{\mathbb T^d}\times \mathbb S^{d-1}}(Id-s\otimes s):\nabla B(x)\,\d \mu_t^h(x,s)=\int_{\mathbb T^d} \xhi^h(t)\text{div}(u_h(\cdot,t)B)\,\d x,\label{allard2}
\end{align}
for all $t\in(0,\infty)$ and all $B\in S_{\xhi^{h}(t)}$. Note that, in the case of volume preserving MCF, we can substitute $w_h$ and $u_h$ with $w_h^0$ and $u_h^0$, that are the zero integral mean projections of $w_h,u_h$. In this volume-preserving framework we can thus pass from $B\in S_{\xhi_{DG}^h(t)}$ and $B\in S_{\xhi^h(t)}$ to any vector field $B\in C^1(\om)$ recalling \cite[Lemma 9]{SH}, which guarantees the existence of the sequences $\lambda_{DG}^h$ and $\lambda_h$ in $L^2(0,\Ts)$ such that, using the energy inequality \eqref{dissiprel} and \eqref{regvbis},
\begin{align}
&\norm{\lambda_{DG}^h}_{L^2(0,\Ts)}\leq C\left(E[\xhi_0], \int_\om \xhi_0\dx\right)(1+\norm{\nabla w_h}_{L^2(0,\Ts;\bL^2(\om))}),\label{lambdas1}\\&
\norm{\lambda_h}_{L^2(0,\Ts)}\leq C\left(E[\xhi_0], \int_\om \xhi_0\dx\right)(1+\norm{\nabla u_h}_{L^2(0,\Ts;\bL^2(\om))}),
    \label{lambdas2}
\end{align}
for any $\Ts>0$.
Then, we have
\begin{align}
	\int_{{\mathbb T^d}\times \mathbb S^{d-1}}(Id-s\otimes s):\nabla B(x)\,\d \widetilde\mu_t^h(x,s)=\int_{\mathbb T^d} \xhi_{DG}^h(t) \text{div}((w_h^0(\cdot,t)+\lambda_{DG}^h)B)\, \d x, \quad \forall B\in C^1(\om),\label{allardvolpres}
\end{align}
as well as 
\begin{align}
	\int_{{\mathbb T^d}\times \mathbb S^{d-1}}(Id-s\otimes s):\nabla B(x)\,\d \mu_t^h(x,s)=\int_{\mathbb T^d} \xhi^h(t)\text{div}((u_h^0(\cdot,t)+\lambda_h)B)\,\d x,\quad \forall B\in C^1(\om),\label{allard2volpres}
\end{align}
also in the case of volume preserving MCF.
From now on we will omit for simplicity the dependence on $E[\xhi_0]$ of all the constants.

\subsection{Limit as $h\to0$}
\subsubsection{Compactness as $h\to0$}\label{compact}
From \eqref{regvbis} we deduce that there exist $ \tw\in L^2_{loc}([0,\infty);H^1({\mathbb T^d}))$ and $u\in L^2_{loc}([0,\infty);H^1({\mathbb T^d}))$ (thus globally defined on $(0,\infty)$), such that, up to subsequences, as $h\to0$, 
\begin{align}
\label{whconv}	w_h\rightharpoonup  \tw &\quad\text{weakly in }L^2(0,\Ts;\Hz),\\
	u_h\rightharpoonup u &\quad\text{weakly in }L^2(0,\Ts;\Hz),\label{uhA}
\end{align}
for any $\Ts>0$. Analogously, in the case of volume preserving MCF, from \eqref{lambdas1}-\eqref{lambdas2} we deduce the existence of $\lambda_1,\lambda_2\in L^2_{loc}(0,\infty)$, possibly different from one another, such that 
\begin{align}
\label{lambdaconv1}	\lambda_{DG}^h\rightharpoonup  \lambda_1\quad &\text{weakly in }L^2(0,\Ts),\\
	\lambda_h\rightharpoonup \lambda_2 \quad&\text{weakly in }L^2(0,\Ts),\label{lambdaconv2}
\end{align}
for any $\Ts>0$.
 
 Proceeding in the compactness argument, since $E[\xhi_{DG}^h(t)]$ is monotone nonincreasing by \eqref{monotonicity}, and uniformly bounded by \eqref{dissiprel}, by Helly's selection theorem there exists a nonincreasing function $\tF$ such that, up to subsequences,
 	\begin{align}
 		E[\xhi_{DG}^h(t)]\to \tF(t)\quad \text{as }h\to 0,\quad \forall t\geq 0.\label{Helly}
 	\end{align}
This will be needed to identify the limit of $\normmm{\mu_t^h}_{\mathbb S^{d-1}}({{\mathbb T^d}})$ as $h\to0$. Analogously, concerning $E[\xhi^h(t)]$, recalling \eqref{monotonicity2} and \eqref{fundamental}, Helly's selection theorem also gives that there exists a nonincreasing function $\tF_1$ such that, up to subsequences,
 	\begin{align}
 		E[\xhi^h(t)]\to \tF_1(t)\quad \text{as }h\to 0,\quad \forall t\geq 0.\label{Helly2}
 	\end{align} 
As a consequence, passing to the limit in \eqref{fundamental}, we infer that, for any $t\geq0$,
\begin{align*}
    \tF(t)\leq \tF_1(t)\leq \tF(s),\quad \forall 0\leq s<t,
\end{align*}
so that, if $t\geq0$ is a continuity point of $\tF$, we infer $\tF(t)= \tF_1(t)$ by letting $s\nearrow t^-$. Since $\tF$ is monotone nonincreasing, the set of (jump) discontinuities is countable and thus has zero $\mathcal L^1$ measure. We can then deduce 
\begin{align}
\tF(t)=\tF_1(t),\quad \text{ for almost any }t\geq 0.
    \label{identification}
\end{align}
We now need to prove a compactness result for the phase  functions, in order to show that they all converge to the same quantity. We use the well-known Aubin-Lions-Simon compactness theorem. Let us first observe that, for any $\delta\in(0,\Ts)$, 
\begin{align}
	\int_0^{\Ts-\delta}\Ld{\wchi^h(t+\delta)-\wchi^h(t)}^2\dt \leq  C\delta,\quad \forall h\in(0,1),
	\label{control1b}
\end{align}
where $C>0$ does not depend on $h$. Indeed, we can write, by the Fundamental Theorem of Calculus and the Cauchy-Schwarz inequality,
\begin{align}
\nonumber\int_0^{\Ts-\delta}\Ld{\wchi^h(t+\delta)-\wchi^h(t)}^2\dt &\leq 
	 \int_0^{\Ts-\delta}\int_{t}^{t+\delta}\delta\Ld{\pt \wchi^h(\tau)}^2\d\tau\dt\\&
	 \leq \delta(\Ts-\delta)
	\int_{0}^{	\Ts}\Ld{\pt \wchi^h(\tau)}^2\d\tau
	\nonumber
    \\&
    \leq C_A(\Ts)\delta,\label{ineq1}
\end{align}
recalling \eqref{regvbis}, for some $C_A(\Ts)$ independent of $h$.

Also, notice that, for almost any $t\in(0,\Ts)$, for $n=\floor{\tfrac th}$,
\begin{align}
&	\nonumber\Ld{\chi^h(t)-{\wchi}^h(t)}=\Ld{\chi_n^h-\wchi^h(t)}=\Ld{\wchi^h(nh)-\wchi^h(t)}\\&
	\leq \int_{nh}^{t}\Ld{\pt\wchi^h(\tau)}\d\tau\leq \sqrt{t-nh}\left(\int_0^\Ts \Ld{\pt\wchi^h(\tau)}^2\right)^\onehalf\leq C(\Ts)\sqrt h,
	\label{relation}
\end{align}
entailing
\begin{align}
	\int_0^{\Ts}\Ld{\chi^h(t)-{\wchi}^h(t)}^2\dt\leq C_B(\Ts)h,
	\label{controla}
\end{align}
for some $C_B(\Ts)$ independent of $h$.
Let us now fix $\epsilon>0$ and choose $\wtilde h_1>0$ such that $C_A(\Ts)\wtilde h_1\leq \epsilon$. Consider a sequence $h_n\to 0$. 
Since the right translation operator is continuous in $L^2(0,T^*-h;\dual)$, we can without loss of generality assume that $h_n\leq \widetilde h_1$ for any $n\in\N$.

Then we can find a sufficiently small $\delta_1<\tfrac \epsilon{C_A}$, such that, by \eqref{ineq1},
\begin{align*}
	&\int_0^{\Ts-\delta}\Ld{\chi^{h}(t+\delta)-\chi^{h}(t)}^2\dt\\&
	\leq 3\int_0^{\Ts-\delta}\Ld{\chi^{h}(t+\delta)-\wchi^{h}(t+\delta)}^2\dt+3\int_0^{\Ts-\delta}\Ld{\chi^{h}(t)-\wchi^{h}(t)}^2\dt\\&
	~~~+3\int_0^{\Ts-\delta}\Ld{\wchi^{h}(t+\delta)-\wchi^{h}(t)}^2\dt\\&
	\leq 9\epsilon,\quad \forall h\leq \wtilde h_1,\quad \forall 0<\delta<\delta_1.
\end{align*}
This entails
\begin{align}
	\sup_{n\in \mathbb N}\int_0^{\Ts-\delta}\Ld{\chi^{h_n}(t+\delta)-\chi^{h_n}(t)}^2\dt\to 0\quad \text{ as }\delta\to0.
	\label{controlB}
\end{align}
Concerning $\chi_{DG}^h(t)$, observe that it holds, recalling \eqref{energy1abc} and \eqref{equiv},
\begin{align}
\nonumber	&\int_0^\Ts\Ld{\xhi_{DG}^h(t)-\xhi^h(t)}^2\dt \\&\leq \nonumber	\int_0^\Ts\Th{t}^2\Ld{\frac{1}{T_h(t)}\left(\xhi_{DG}^h(t)-\xhi^h(t)\right)}^2\dt\nonumber\\&
	\leq C(\Ts)h^2,\label{CC}
\end{align}
 where $C_C(\Ts)$ does not depend on $h$. Now fix $\epsilon>0$, choose $\wtilde h_2>0$ such that $C_C(\Ts)\wtilde h_2^2\leq \epsilon$, and consider a sequence $h_n\to 0$. Again we can assume without loss of generality that $h_n\leq \wtilde h_2$ for any $n\in \N$. 

As a consequence, fixing the same subsequence as the one in \eqref{controlB}, by the triangle inequality, \eqref{controlB}, and \eqref{CC}, we deduce  
\begin{align*}
	\sup_{h_n\leq \wtilde h_2}\int_0^{\Ts-\delta}\Ld{\chi^{h_n}_{DG}({t+\delta})-\chi^{h_n}_{DG}({t})}^2\dt\leq 9\epsilon,
\end{align*}
for any $\delta>0$ sufficiently small, which implies,
\begin{align}
	\sup_{n\in\mathbb N}\int_0^{\Ts-\delta}\Ld{\chi^{h_n}_{DG}({t+\delta})-\chi^{h_n}_{DG}({t})}^2\dt\to 0\quad \text{ as } \delta\to0.\label{subseq1}
\end{align}
Now, by \eqref{energy1ab} and \eqref{energy1abc}, we also have 
\begin{align*}
	\wchi^h(\cdot),\xhi^h(\cdot),\xhi_{DG}^h(\cdot)	\in L^\infty_{w*}(0,\Ts;BV(\om;\{0,1\})),
\end{align*}
with norms uniformly bounded in $h$. Therefore, by the embedding $$BV(\om;\{0,1\})\hookrightarrow\hookrightarrow L^{p}({\mathbb T^d})\hookrightarrow \dual,$$ where $p\in(\tfrac{2d}{d+2},\tfrac{d}{d-1})$ (since $d=2,3$), recalling \eqref{control1b}, \eqref{controlB}, \eqref{subseq1}, we can apply the Aubin-Lions-Simon Theorem, together with a diagonal argument, to infer the existence of $\xhi\in L^2_{loc}([0,\infty);L^1({\mathbb T^d}))$, such that, up to a subsequence,
\begin{align*}
	\wchi^h(\cdot),\xhi^h(\cdot),\xhi_{DG}^h(\cdot)\to \xhi(\cdot) \text{ strongly in }L^2(0,\Ts;L^1({\mathbb T^d})),
\end{align*}
for any $\Ts>0$, where the identification of the three limits has been possible thanks to \eqref{controla} and \eqref{CC}. This also entails, up to another subsequence, 
\begin{align*}
	\wchi^h(\cdot),\xhi^h(\cdot),\xhi_{DG}^h(\cdot)\to \xhi(\cdot) \text{ for almost any   }(t,x)\in (0,\Ts)\times{\mathbb T^d},
\end{align*}
which entails that $\xhi\in \{0,1\}$ almost everywhere in $(0,\infty)\times\om$.
Moreover, as all the functions are bounded in $[0,1]$, we infer 
\begin{align}
		\wchi^h(\cdot),\xhi^h(\cdot),\xhi_{DG}^h(\cdot)\to \xhi(\cdot)\text{ strongly in }L^p((0,\Ts)\times\om),\quad \forall p\in[1,\infty).
	\label{finalconv}
\end{align}
As noticed in \cite{AbelsRoger, AbelsP}, since $\xhi^h$ is uniformly bounded in $L^\infty_{w*}(0,\infty;BV(\om;\{0,1\}))$, this implies that $\xhi_h\to \xhi$ weakly* in  $L^\infty_{w*}(0,\infty;BV(\om;\{0,1\}))$, and thus, $\xhi\in L^\infty_{w*}(0,\infty;BV(\om;\{0,1\}))$.

Also, as $\pt \wchi^h$ is uniformly bounded in $L^2(0,\infty;\dual)\hookrightarrow L^2(0,\infty;H^2(\om)')$ by \eqref{regvbis}, we deduce $\xhi\in H^1(0,\infty;H^2(\om)')$, and thus, since this space is reflexive,
\begin{align}
\pt \wchi^h\rightharpoonup \pt \xhi\quad \text{ weakly in }L^2(0,\Ts;H^2(\om)'),
\label{dtchi1}
\end{align}
for any $\Ts>0$. 
Also, an adaptation of \cite[Lemma C.1]{AbelsP}, substituting $H^1({\mathbb T^d})'$ with $H^2({\mathbb T^d})'$, gives 
\begin{align}
	\xhi\in C([0,\Ts];L^p(\om))\quad \forall \Ts>0,\quad 	\forall p\in [1,\infty).
	\label{regchi1a}
\end{align}

\subsubsection{Limit varifolds I. Disintegration of the measure} 

Thanks to the uniform bound of the energy \eqref{dissiprel}, by the weak*-compactness properties of finite Radon measures there exists a measure $\widetilde{\mu}:{\mathcal B((0,\infty)	\times{\mathbb T^d}\times \mathbb S^{d-1})}\to \mathbb R^+$, such that, up to subsequences, 
\begin{align}
	\mathcal L^1\llcorner{(0,\Ts)}\otimes (\wtilde\mu_h^t)_{t\in(0,\Ts)}\rightharpoonup \widetilde{\mu},\quad \text{ weakly* in } \mathcal M((0,\Ts)\times{\mathbb T^d}\times \mathbb S^{d-1}),\label{convmeas}
\end{align}
 for any $\Ts>0$, where $\mathcal B((0,\infty)	\times{\mathbb T^d}\times \mathbb S^{d-1})$ is the Borel $\sigma$-algebra on $(0,\infty)	\times{\mathbb T^d}\times \mathbb S^{d-1}$. This measure $\widetilde\mu$ can be sliced in time thanks to the convergence \eqref{Helly}, following, e.g., \cite{SH, AbelsP}. 
In particular, there
exists an $\mathcal L^1\llcorner(0,\infty)$-measurable function  $\widetilde\sigma$ on $(0, \infty)$ and a weakly-* $\sigma$-measurable family of Radon probability measures $(\tau_t)_t$ on ${\mathbb T^d}\times\mathbb S^{d-1}$, such that
\begin{align}
	\widetilde{\mu}=\mathcal L^1\llcorner(0,\infty)\otimes (\widetilde\mu_t)_{t\in(0,\infty)},
	\label{representation}
\end{align}
with $\widetilde\mu_t=\widetilde{\sigma}(t)\tau_t$.
By the very same arguments, exploiting this time \eqref{Helly2} and \eqref{identification}, we infer that 
there exists $\mu:{\mathcal B((0,\infty)	\times{\mathbb T^d}\times \mathbb S^{d-1})}\to \mathbb R^+$, such that, up to subsequences, 
\begin{align}
	\mathcal L^1\llcorner{(0,\Ts)}\otimes (\mu_h^t)_{t\in(0,\Ts)}\rightharpoonup \mu,\quad \text{ weakly* in } \mathcal M((0,\Ts)\times{\mathbb T^d}\times \mathbb S^{d-1}),\label{convmeas1}
\end{align}
 for any $\Ts>0$, where 
\begin{align}
	\mu=\mathcal L^1\llcorner(0,\infty)\otimes (\mu_t)_{t\in(0,\infty)},
	\label{representation1}
\end{align}
and $\mu_t$ is weakly-* $\mathcal L^1\llcorner(0,\infty)$ measurable. Then (see also, e.g., \cite[Lemma 2]{LH}), recalling the definition of $\mu_h^t$, it holds, exploiting \eqref{Helly2},
\begin{align*}
     \normmm{\mu_t}_{\mathbb S^{d-1}}({\mathbb T^d})\leq \liminf_{h\to\infty}E[\xhi^h(t)]
\end{align*}
for almost any $t>0$. Furthermore, considering the set $[0,\Ts]\times {\mathbb T^d}\times \mathbb S^{d-1}$ which is compact in $[0,\infty)\times{\mathbb T^d}\times \mathbb S^{d-1}$, we infer by \cite[Proposition 4.26]{Maggi}
\begin{align}
   \limsup_{h\to\infty} \int_0^\Ts E[\xhi^h(t)]\dt\leq \int_0^\Ts \normmm{\mu_t}_{\mathbb S^{d-1}}({\mathbb T^d})\dt
   \label{limsup}
\end{align}
for any $\Ts>0$. Now recalling \eqref{Helly2} and \eqref{identification}, we get
\begin{align*}
      \widetilde E(t)=\lim_{h\to\infty} E[\xhi^h(t)]\geq \normmm{\mu_t}_{\mathbb S^{d-1}}({\mathbb T^d})
\end{align*}
for almost any $t>0$. This gives of course 
\begin{align}
     \int_0^\Ts \widetilde E(t)\dt\geq \int_0^\Ts\normmm{\mu_t}_{\mathbb S^{d-1}}({\mathbb T^d})\dt\label{inf}
\end{align}
for any $\Ts>0$. Also, by Lebesgue's Dominated Convergence Theorem, using again \eqref{Helly2}, \eqref{limsup} entails
\begin{align*}
    \int_0^\Ts \widetilde E(t)\dt=\lim_{h\to\infty} \int_0^\Ts E[\xhi^h(t)]\dt=\limsup_{h\to\infty} \int_0^\Ts E[\xhi^h(t)]\dt\leq \int_0^\Ts \normmm{\mu_t}_{\mathbb S^{d-1}}({\mathbb T^d})\dt
\end{align*}
for any $\Ts>0$, which combined with \eqref{inf} gives
\begin{align*}
    \int_0^\Ts \widetilde E(t)\dt= \int_0^\Ts \normmm{\mu_t}_{\mathbb S^{d-1}}({\mathbb T^d})\dt,
\end{align*}
entailing, since $\Ts$ is arbitrary, 
\begin{align}
    \widetilde E(t)=\normmm{\mu_t}_{\mathbb S^{d-1}}({\mathbb T^d})\label{identif}
\end{align}
for almost any $t>0$.

In the following we need to retrieve information on $\normmm{\rho_h\ast \nabla \xhi^h}$, which is apparently not easy to be treated. Nevertheless, as it will be clear in the sequel (see, e.g., \eqref{upperbound} below), it is enough to obtain a suitable upper bound on this measure. We thus restrict our attention to the simpler positive measure $\rho_h\ast \normmm{\nabla \xhi^h}$, noticing that, by \eqref{convmeas1}, we have,  since $\mathbb S^{d-1}$ is compact,  
\begin{align}
	\mathcal L^1\llcorner{(0,\Ts)}\otimes (\normmm{\mu_h^t}_{\mathbb S^{d-1}})_{t\in(0,\Ts)}\rightharpoonup \mathcal L^1\llcorner{(0,\Ts)}\otimes (\normmm{\mu_t}_{\mathbb S^{d-1}})_{t\in(0,\Ts)}\quad \text{ weakly* in } \mathcal M((0,\Ts)\times{\mathbb T^d}),\label{convmeas2}
\end{align}
entailing
\begin{align}
	\mathcal L^1\llcorner{(0,\Ts)}\otimes (\rho_h\ast\vert{\nabla\xhi^h(t)}\vert)_{t\in(0,\Ts)}\rightharpoonup \mathcal L^1\llcorner{(0,\Ts)}\otimes (\normmm{\mu_t}_{\mathbb S^{d-1}})_{t\in(0,\Ts)}\quad \text{ weakly* in } \mathcal M((0,\Ts)\times{\mathbb T^d}).\label{convmeas3}
\end{align}
Here we made a slight abuse of notation, writing $$\rho_h\ast \normmm{\nabla \xhi^h(t)}(A)=\int_A \rho_h\ast \normmm{\nabla \xhi^h(t)}(x)\dx$$ for any Borel measurable set $A\subset \R^d$. Analogously, we set $$\normmm{\rho_h\ast {\nabla \xhi^h(t)}}(A):=\int_A \normmm{\rho_h\ast {\nabla \xhi^h(t)}}(x)\dx$$ for any Borel measurable set $A\subset \R^d$.
To prove \eqref{convmeas3}, given $\varphi\in C_c((0,\Ts)\times {\mathbb T^d})$ we get 
\begin{align*}
    \int_0^\Ts\int_{\mathbb T^d} \varphi(x,t)\rho_h\ast\normmm{\nabla \xhi^h(t)}\dx\dt= \int_0^\Ts\int_{\mathbb T^d} \rho_h\ast\varphi(x,t)\, \d\normmm{\nabla \xhi^h(t)}\dt,
\end{align*}
and we have 
\begin{align*}
&\normmm{\int_0^\Ts\int_{\mathbb T^d} \rho_h\ast\varphi(x,t)\,\d\normmm{\nabla \xhi^h(t)}\dt}\\&\leq \normmm{\int_0^\Ts\int_{\mathbb T^d} \varphi(x,t)\,\d\normmm{\nabla \xhi^h(t)}\dt-\int_0^\Ts\int_{\mathbb T^d} \varphi(x,t)\,\d\normmm{\mu_t}_{\mathbb S^{d-1}}\dt}\\&\quad +\normmm{\int_0^\Ts\int_{\mathbb T^d} (\rho_h\ast\varphi(x,t)-\varphi(x,t))\, \d\normmm{\nabla \xhi^h(t)}\dt}\to 0\quad\text{ as }h\to0,
\end{align*}
since the first term on the right-hand side converges to zero by \eqref{convmeas2}, whereas the second one converges as, recalling \eqref{fundamental} and \eqref{dissiprel},
\begin{align*}
 &\normmm{\int_0^\Ts\int_{\mathbb T^d} (\rho_h\ast\varphi(x,t)-\varphi(x,t))\,\d\normmm{\nabla \xhi^h(t)}\dt}\\&\leq \normmm{\int_0^\Ts\norm{\rho_h\ast \varphi(t)-\varphi(t)}_{C(K)}\normmm{\nabla\xhi^h(t)}(\om)\dt}\\&
 \leq \int_0^\Ts\norm{\rho_h\ast \varphi(t)-\varphi(t)}_{C(K)} E[\xhi_0]\dt,
\end{align*}
with $\text{supp } \varphi \subset [0,\Ts]\times K$, $K\subset \R^d$ compact, and the right-hand converges to zero by Lebesgue's Dominated Convergence Theorem, since $\rho_h\ast \varphi(t)\to \varphi(t)$ uniformly on $K\subset {\mathbb T^d}$ compact, for any $t\in (0,\Ts)$.

\subsubsection{Gibbs-Thomson law as $h\to0$ and generalized mean curvature.}
\label{gibbs}
For any $\Ts>0$, let us consider a vector field $ B:(0,\Ts)\times\om\to \R^d$ continuous and compactly supported in $ (0,\Ts)\times\om$, such that, for any $t\in(0,\Ts)$, $B(\cdot,t)\in C^1({\mathbb T^d},\mathbb R^d)$.

Then by \eqref{convmeas} and \eqref{representation} it holds 
\begin{align}
	\int_0^\Ts \delta\widetilde\mu_t^h(B)\dt\to \int_0^\Ts \delta\widetilde{\mu}_t(B)\dt. \label{v1}
\end{align}
We now distinguish between the MCF case and the volume preserving MCF, which are slightly different.
\begin{itemize}
\item (MCF). We multiply the Gibbs-Thomson relation \eqref{allard} by a smooth and
compactly supported test function on $(0, \Ts)$, integrate in time, and pass to the limit as $h\to0$, by means of \eqref{whconv} and \eqref{finalconv}, and then, after again a localization in time, deduce from \eqref{allard} that, for almost any  $t>0$, it holds 
\begin{align}
	&	\delta\wtilde\mu_t(B)=\int_{{\mathbb T^d}\times \mathbb S^{d-1}}(Id-s\otimes s):\nabla B(x)\,\d \widetilde{\mu}_t(x,s)\nonumber\\&=
	\int_{\mathbb T^d} \chi(\cdot,t)\text{div}(\tw(\cdot,t)B)\,\d x,\label{variation2}
\end{align} 
for any $B\in C^1({\mathbb T^d};\R^d)$.

Analogously, the uniform control \eqref{regvbis} on $u_h$, and \eqref{allard2} give 
\begin{align}
	&	\delta\mu_t(B)=\int_{{\mathbb T^d}\times \mathbb S^{d-1}}(Id-s\otimes s):\nabla B(x)\,\d \mu_t(x,s)\nonumber\\&=
	\int_{\mathbb T^d} \chi(\cdot,t)\text{div}(u(\cdot,t)B)\,\d x,\label{variation2b}
\end{align} 
for any $B\in C^1({\mathbb T^d};\R^d)$ and almost any $t>0$. 

\item (volume preserving MCF). We multiply the Gibbs-Thomson relation \eqref{allardvolpres} by a smooth and
compactly supported test function on $(0, \Ts)$, integrate in time, and pass to the limit as $h\to0$, by means of \eqref{whconv}, \eqref{lambdaconv1}, and \eqref{finalconv}, to deduce from \eqref{allardvolpres} that, for almost any  $t>0$, it holds 
\begin{align}
	&	\delta\wtilde\mu_t(B)=\int_{{\mathbb T^d}\times \mathbb S^{d-1}}(Id-s\otimes s):\nabla B(x)\,\d \widetilde{\mu}_t(x,s)\nonumber\\&=
	\int_{\mathbb T^d} \chi(\cdot,t)\text{div}((\tw^0(\cdot,t)+\lambda_1(t))B)\,\d x,\label{variation2volpres}
\end{align} 
for any $B\in C^1({\mathbb T^d};\R^d)$, where $w^0$ is the zero integral mean projection of $w$. With the same argument, the uniform control \eqref{regvbis} on $u_h$, and \eqref{allard2volpres} allow to infer 
\begin{align}
	&	\delta\mu_t(B)=\int_{{\mathbb T^d}\times \mathbb S^{d-1}}(Id-s\otimes s):\nabla B(x)\,\d \mu_t(x,s)\nonumber\\&=
	\int_{\mathbb T^d} \chi(\cdot,t)\text{div}((u^0(\cdot,t)+\lambda_2(t))B)\,\d x,\label{variation2volpres1}
\end{align} 
for any $B\in C^1({\mathbb T^d};\R^d)$ and almost any $t>0$.
\end{itemize}
Following then \cite[Lemmas 4.1-4.2]{Roger}, we can prove that, for any $B\in L^2(0,\Ts; C^1({\mathbb T^d};\R^d))$,
\begin{align}
	\int_0^\Ts \delta\widetilde\mu_t^h(B)\,\d t \to -\int_0^\Ts\int_{\mathbb T^d} H_{\xhi(\cdot,t)}\frac{\nabla\xhi(\cdot,t)}{\normmm{\nabla\chi(\cdot,t)}}\cdot B\,\d \normmm{\nabla\xhi(\cdot,t)}\dt,\label{cona}
\end{align}
for any $\Ts>0$, where $H_\xhi$ is the generalized mean curvature in the sense of \cite{Roger}, which is intrinsic to the surface $\text{supp}\normmm{\nabla\xhi(\cdot,t)}$ for almost every $t>0$. Note that this property will not be preserved in the limit as $\alpha\to0$. 

Again, thanks to \eqref{convmeas1}, the uniform control \eqref{regvbis} on $u_h$, and \eqref{allard2}, the same argument leads to 
\begin{align}
	\int_0^\Ts \delta\mu_t^h(B) \,\d t \to -\int_0^\Ts\int_{\mathbb T^d} H_{\xhi(\cdot,t)}\frac{\nabla\xhi(\cdot,t)}{\normmm{\nabla\chi(\cdot,t)}}\cdot B\d \normmm{\nabla\xhi(\cdot,t)}\dt\label{conab}
\end{align}
for any $B\in L^2(0,\Ts;C^1(\om))$. It is then immediate to infer that 
\begin{align}
	&	\delta\wtilde\mu_t(B)=\delta \mu_{t}(B),\label{variation4}
\end{align} 
for any $B\in C^1({\mathbb T^d};\R^d)$ and for almost any $t>0$. Note that this identification is crucial to prove the validity of a suitable energy inequality. This is the only reason why we need to limit ourselves to dimensions $d=2,3$, in order to apply the fundamental results of \cite{Schatzle}. Namely, these results require the chemical potential at almost any time to belong to $W^{1,p}(\om)$ with $p>\tfrac d2$, which would be $p>2$ if $d>3$, and this control is not possible in this framework. Also, this identification is the reason why we need to keep $\alpha>0$ fixed and first let $h\to0$, since it only holds at the level of the Mullins-Sekerka-type approximation, as it is related to the higher integrability of the generalized mean curvature vector given by the chemical potential.

\subsubsection{Limits of the dissipative terms as $h\to0$: mean curvature term related to $w_h$}

We first need to consider the term which will be associated to the generalized mean curvature, namely we can write, by Young's inequality,
\begin{align}
    \nonumber&\frac12\int_0^\Ts\int_{\mathbb T^d} w_h(t)^2\normmm{\rho_h\ast{\nabla \xhi^h(t)}}\dx\dt\\&\geq \int_0^\Ts\int_{\mathbb T^d} w_h(t)\zeta\normmm{\rho_h\ast{\nabla \xhi^h(t)}}\dx\dt-\frac 12\int_0^\Ts\int_{\mathbb T^d} \zeta^2\normmm{\rho_h\ast{\nabla \xhi^h(t)}}\dx\dt,\label{inequality}
\end{align}
for $\zeta=\xi\cdot\frac{\rho_h\ast\nabla\xhi^h(t)}{\normmm{\rho_h\ast \nabla\xhi^h(t)}}$, with $\xi\in C^\infty_c((0,\Ts)\times{\mathbb T^d};\R^d)$ and $\xi(t)\in S_{\xhi_t}$ for any $t\in(0,\Ts)$, \textcolor{black}{where we tacitly set $\frac{\rho_h\ast\nabla\xhi^h(t)}{\normmm{\rho_h\ast \nabla\xhi^h(t)}}=0$ when $\normmm{\rho_h\ast \nabla\xhi^h(t)}=0$}. Note that the condition of $\xi(t)\in S_{\xhi_t}$ has a role only in the case of volume preserving MCF. Therefore, since $\normmm{\frac{\rho_h\ast\nabla\xhi^h(t)}{\normmm{\rho_h\ast \nabla\xhi^h(t))}}}=1$, we have
\begin{align}
    &\nonumber\frac12\int_0^\infty\int_{\mathbb T^d} w_h(t)^2\normmm{\rho_h\ast{\nabla \xhi^h(t)}}\dx\dt\\&\geq {\int_0^\Ts\int_{\mathbb T^d} w_h(t)\xi(t)\cdot\rho_h\ast{\nabla \xhi^h(t)}}\dx\dt-\frac 12\int_0^\Ts\int_{\mathbb T^d} \normmm{\xi(t)}^2\normmm{\rho_h\ast{\nabla \xhi^h(t)}}\dx\dt.\label{curvat}
\end{align}
Now, thanks to the properties of the mollifier $\rho_h$, it holds $\nabla(\rho_h\ast \xhi^h)=\rho_h\ast\nabla\xhi^h$, and thus we deduce, after using the Gau{\ss} Theorem, that
\begin{align}
    {\int_0^\Ts\int_{\mathbb T^d} w_h(t)\xi(t)\cdot\rho_h\ast{\nabla \xhi^h(t)}}\dx\dt=-{\int_0^\Ts\int_{\mathbb T^d} \Div(w_h(t)\xi(t))\rho_h\ast{ \xhi^h(t)}}\dx\dt.\label{identity}
\end{align}
Observe that it holds, recalling \eqref{finalconv} and the properties of the mollifier $\rho_h$,
\begin{align}
&\nonumber\norm{\rho_h\ast \xhi^h-\xhi_{DG}^h}_{L^p((0,\Ts)\times{\mathbb T^d})}\nonumber\\&\leq\nonumber \norm{\rho_h\ast (\xhi^h-\chi)}_{L^p((0,\Ts)\times{\mathbb T^d})}+\norm{\rho_h\ast\xhi-\chi}_{L^p((0,\Ts)\times{\mathbb T^d})}+\norm{ \xhi_{DG}^h-\chi}_{L^p((0,\Ts)\times{\mathbb T^d})}\nonumber\\&
\leq 
\norm{\xhi^h-\chi}_{L^p((0,\Ts)\times{\mathbb T^d})}
+\norm{\rho_h\ast\xhi-\chi}_{L^p((0,\Ts)\times{\mathbb T^d})}+\norm{ \xhi_{DG}^h-\chi}_{L^p((0,\Ts)\times{\mathbb T^d})}\to 0\text{ as }h\to0,
    \label{convergenceA}
\end{align}
for any $p\in[1,\infty)$.
Then, using \eqref{regvbis} and \eqref{convergenceA}, by Cauchy-Schwarz inequality we infer
\begin{align*}
 & \normmm{{\int_0^\Ts\int_{\mathbb T^d} \Div(w_h(t)\xi(t))\rho_h\ast{ \xhi^h(t)}}\dx\dt-{\int_0^\Ts\int_{\mathbb T^d} \Div(w_h(t)\xi(t)){ \xhi_{DG}^h(t)}}\dx\dt} \\&
 \leq C\norm{w_h}_{L^2(0,\Ts;H^1({\mathbb T^d}))}\norm{\xi}_{L^\infty(0,\Ts;C^1(\om))}\norm{\rho_h\ast\xhi^h-\xhi_t^{\Th{\cdot}}}_{L^{2}((0,\Ts)\times{\mathbb T^d})} \\&
 \leq \frac C{\alpha} \norm{\xi}_{L^\infty(0,\Ts;C^1(\om))}\norm{\rho_h\ast\xhi^h-\xhi_t^{\Th{\cdot}}}_{L^d((0,\Ts)\times{\mathbb T^d})}\to 0\quad\text{ as }h\to0,
\end{align*}
for any $\Ts>0$, where we exploited the convergence properties of the mollifier and the regularity of the (fixed) vector field $\xi$.

Moreover, using the convergences \eqref{whconv} and \eqref{finalconv}, together with the properties of the mollifiers, we immediately deduce that
\begin{align*}
    \int_0^\Ts\int_{\mathbb T^d} \Div(w_h(t)\xi(t))\rho_h\ast{ \xhi^h(t)}\dx\dt\to \int_0^\Ts\int_{\mathbb T^d} \Div(w(t)\xi(t)){ \xhi(t)}\dx\dt\quad\text{ as }h\to 0.
\end{align*}
Therefore, since $\xi(t)\in S_{\xhi_t}$ for any $t\in(0,\Ts)$, using the Gibbs-Thomson law \eqref{variation2} (cf. \eqref{variation2volpres1} for the volume preserving MCF), together with the crucial identification \eqref{variation4}, we can identify the limit, obtaining 
\begin{align*}
    \int_0^\Ts\int_{\mathbb T^d} \Div(w_h(t)\xi(t))\rho_h\ast{ \xhi^h(t)}\dx\dt\to \int_0^\Ts\delta\mu_t(\xi(t))\dt\quad\text{ as }h\to 0
\end{align*}
for any $\xi\in C_c^1((0,\Ts);S_{\xhi_{(\cdot)}})$.

We are left with the identification of the limit of the second term in the inequality \eqref{curvat}. To this aim, as apparently it is not trivial to prove that $\normmm{\rho_h\ast\nabla\xhi^h}$ weakly-* converges as measures, we exploit \eqref{convmeas3} and deduce, since $\rho_h$ is a positive mollifier,
\begin{align}
    \frac 12\int_0^\Ts\int_{\mathbb T^d} \normmm{\xi(t)}^2\normmm{\rho_h\ast{\nabla \xhi^h(t)}}\dx\dt\leq \frac12 \int_0^\Ts\int_{\mathbb T^d} \normmm{\xi(t)}^2\rho_h\ast\normmm{{\nabla \xhi^h(t)}}\dx\dt,\label{upperbound}
\end{align}
and this term indeed converges by \eqref{convmeas3}, namely
\begin{align*}
    \frac12 \int_0^\Ts\int_{\mathbb T^d} \normmm{\xi(t)}^2\rho_h\ast\normmm{{\nabla \xhi^h(t)}}\dx\dt\to \frac12 \int_0^\Ts\int_{\mathbb T^d} \normmm{\xi(t)}^2\d\normmm{\mu_t}_{\mathbb S^{d-1}}\dt.
\end{align*}
This means that, from \eqref{curvat} and \eqref{upperbound}, we can infer
\begin{align}
    &\nonumber\frac12\int_0^\Ts\int_{\mathbb T^d} w_h(t)^2\normmm{\rho_h\ast{\nabla \xhi^h(t)}}\dx\dt\\&\geq {\int_0^\Ts\int_{\mathbb T^d} w_h(t)\xi(t)\cdot\rho_h\ast{\nabla \xhi^h(t)}}\dx\dt-\frac 12\int_0^\Ts\int_{\mathbb T^d} \normmm{\xi(t)}^2{\rho_h\ast\normmm{\nabla \xhi^h(t)}}\dx\dt, \label{curvat2}
\end{align}
and taking the $\liminf_{h\to0}$ of the two sides of the inequality we get
\begin{align}
    &\nonumber\liminf_{h\to0}\frac12\int_0^\Ts\int_{\mathbb T^d} w_h(t)^2\normmm{\rho_h\ast{\nabla \xhi^h(t)}}\dx\dt\\&\geq \int_0^\Ts \delta\mu_t(\xi(t))\dt-\frac 12 \int_0^\Ts\int_{\mathbb T^d} \normmm{\xi(t)}^2\d\normmm{\mu_t}_{\mathbb S^{d-1}}\dt,\label{curvat4}
\end{align}
for any $\xi\in C_c^\infty((0,\infty)\times{\mathbb T^d};\R^d)$ such that $\xi(t)\in S_{\xhi_t}$ for any $t>0$. Analogously, we can argue on the time intervals $(s,T')$, $0\leq s<T'<\infty$, and obtain the desired
\begin{align}
&\nonumber\liminf_{h\to0}\frac12\int_s^{T'}\int_{\mathbb T^d} w_h(t)^2\normmm{\rho_h\ast{\nabla \xhi^h(t)}}\dx\dt\\&\geq \int_s^{T'} \delta\mu_t(\xi(t))\dt-\frac 12 \int_s^{T'} \norm{\xi(t)}^2_{\mathcal V_{\xhi_t}}\dt,\quad\forall{ \xi\in C^1_c((s,T');S_{\xhi_{(\cdot)}})},\label{curvat5}
\end{align}
where we recall that the inner product on $\mathcal V_{\xhi_t}$ is induced by the $L^2(\om;\d\normmm{\mu_t}_{\sf})$ norm (see \eqref{Vchi}).
Taking now the supremum over $\xi\in C^1_c((s,T');\mathcal S_{\chi_{(\cdot)}})$ gives 
\begin{align}
\liminf_{h\to0}\frac12\int_s^{T'}\int_{\mathbb T^d} w_h(t)^2\normmm{\rho_h\ast{\nabla \xhi^h(t)}}\dx\dt\geq \onehalf\int_s^{T'} \normmm{\partial E(\mu_t)}_{\mathcal V_{\xhi_t}}^2\dt,\label{curvat6}\quad \forall 0\leq s<T',
\end{align}
recalling \eqref{slope}.

Concerning the other dissipative terms in $w_h$, we simply have by lower semicontinuity due to \eqref{whconv}, 
\begin{align*}
 &  \alpha\norm{ w}_{L^2(s,T';L^2({\mathbb T^d}))}^2+  \alpha\norm{\nabla  w}_{L^2(s,T';L^2({\mathbb T^d}))}^2 \\&\leq \liminf_{h\to 0}\alpha\norm{w_h}_{L^2(s,T';L^2({\mathbb T^d}))}^2+  \liminf_{h\to 0}\alpha\norm{\nabla w_h}_{L^2(s,T';L^2({\mathbb T^d}))}^2\\&\leq \liminf_{h\to0}\left( \alpha\norm{w_h}_{L^2(s,T';L^2({\mathbb T^d}))}^2+  \alpha\norm{\nabla w_h}_{L^2(s,T';L^2({\mathbb T^d}))}^2\right)
\end{align*}
for any $0<s<T$.

\subsubsection{Limits of the dissipative terms as $h\to0$: velocity term related to $u_h$}
We now need to pass to study the velocity term appearing in the energy, which is the one related to $u_h$. This is delicate and requires special attention. To this aim, we first introduce the Hilbert space 
$$W:=L^2(0,\infty;H^1({\mathbb T^d}))\cap L^2((0,\infty);L^2({\mathbb T^d},\d\normmm{\mu_{(\cdot)}}_{\mathbb S^{d-1}})),$$
inducing the norm
\begin{align*}
\norm{v}_{W}^2:={\alpha\norm{v}_{L^2(0,\infty;L^2({\mathbb T^d}))}^2+\alpha\norm{\nabla v}_{L^2(0,\infty;L^2({\mathbb T^d}))}^2+\norm{v}_{L^2(0,\infty;L^2({\mathbb T^d},\d\normmm{\mu_{(\cdot)}}_{\mathbb S^{d-1}}))}^2},
\end{align*}
and corresponding inner product
\begin{align*}
    (v,w)_W=\int_0^\infty\int_{\mathbb T^d} vw\d\normmm{\mu_t}_{\mathbb S^{d-1}}\dt+\alpha\int_0^\infty\int_{\mathbb T^d} vw\dx\dt+\alpha\int_0^\infty\int_{\mathbb T^d} \nabla v\cdot\nabla w\dx\dt,\quad \forall v,w\in W.
\end{align*}
\textcolor{black}{Note that we crucially chose $\alpha$ and not $\alpha^2$ in this definition.}
We now set, given $\zeta\in C^\infty_c((0,\infty)\times\om)$, the weak time derivative as 
$$
\langle\pt \xhi,\zeta\rangle:=-\int_0^\infty\int_\om \xhi\pt\zeta\dx\dt.
$$
Using \eqref{timederiv}, the convergence \eqref{finalconv} (after an integration by parts) and \eqref{uhA}, for any $\zeta\in C_c^\infty({\mathbb T^d}\times(0,\Ts))$ we get, using Cauchy-Schwarz inequality,
\begin{align}
&\nonumber\normmm{\langle\partial_t{\xhi}(t)\zeta(t)\rangle}=\normmm{\int_0^{\Ts}\int_{\mathbb T^d}{\xhi}(t)\partial_t\zeta(t)\dx\dt}\\&\nonumber= \liminf_{h\to 0}\normmm{\int_0^{\Ts}\int_{\mathbb T^d}{\widehat\xhi_h}(t)\partial_t\zeta(t)\dx\dt}\\&\nonumber
\leq \liminf_{h\to0}\int_0^\Ts \int_{\mathbb T^d} {u_h(t)\zeta(t)}\normmm{\rho_h\ast \nabla\chi^h(t)}\dx\dt\\&
\quad +\alpha^2\normmm{\int_0^\Ts\int_\om u(t)\zeta(t)\dx\dt} +\alpha^2\normmm{\int_0^{\Ts}\int_{\mathbb T^d} \nabla u(t)\cdot \nabla\zeta(t)\dx\dt}.\label{dtc}
\end{align}
Proceeding in the estimate, we note that, recalling \eqref{hutchinson} and \eqref{convmeas3}, by Cauchy-Schwarz inequality we get
\begin{align*}
   & \liminf_{h\to0}\int_0^\Ts \int_{\mathbb T^d} {u_h(t)\zeta(t)}\normmm{\rho_h\ast \nabla\chi^h(t)}\dx\dt\\&\leq \liminf_{h\to0}\left(\left(\int_0^\Ts \int_{\mathbb T^d} \normmm{u_h(t)}^2\normmm{\rho_h\ast \nabla\chi^h(t)}\dx\dt\right)^\frac12\left(\int_0^\Ts \int_{\mathbb T^d} \normmm{\zeta(t)}^2\normmm{\rho_h\ast \nabla\chi^h(t)}\dx\dt\right)^\frac12\right)\\&
    \leq E[\xhi_0]^\onehalf\liminf_{h\to0}\left(\int_0^\Ts \int_{\mathbb T^d} \normmm{\zeta(t)}^2\normmm{\rho_h\ast \nabla\chi^h(t)}\dx\dt\right)^\frac12\\& \leq E[\xhi_0]^\onehalf\liminf_{h\to0}\left(\int_0^\Ts \int_{\mathbb T^d} \normmm{\zeta(t)}^2\rho_h\ast \normmm{\nabla\chi^h(t)}\dx\dt\right)^\frac12\\&  \leq E[\xhi_0]^\onehalf\left(\int_0^\Ts \int_{\mathbb T^d} \normmm{\zeta(t)}^2\d\normmm{\mu_t}_{\mathbb S^{d-1}}\dt\right)^\frac12.
\end{align*}
Therefore, using again the estimate on $\partial_t\xhi$ above, we get, after using Cauchy-Schwarz inequality also for the other terms in the right-hand side of \eqref{dtc},
\begin{align*}
&\normmm{-\int_0^{\Ts}\int_{\mathbb T^d}{\xhi}(t)\pt \zeta(t)\dx\dt}\\&
\leq E[\xhi_0]^\onehalf\left(\int_0^\Ts \int_{\mathbb T^d} \normmm{\zeta(t)}^2\d\normmm{\mu_t}_{\mathbb S^{d-1}}\dt\right)^\frac12\\&
\quad +\alpha^\frac12\left(\alpha^2{\int_0^\Ts\int_\om \normmm{u(t)}^2\dx\dt}\right)^\frac12\left(\alpha{\int_0^\Ts\int_\om \normmm{\zeta(t)}^2\dx\dt}\right)^\frac12\\&\quad +\alpha^\frac12\left(\alpha^2{\int_0^{\Ts}\int_{\mathbb T^d} \normmm{\nabla u(t)}^2\dx\dt}\right)^\frac12\left(\alpha{\int_0^{\Ts}\int_{\mathbb T^d} \normmm{\nabla \zeta(t)}^2\dx\dt}\right)^\frac12\\&
\leq (1+2\alpha^\frac12)E[\xhi_0]^\frac12 \norm{\zeta}_W.
\end{align*}
Extending by density $\partial_t\chi\in W'$, since $\Ts$ and $\zeta$  are arbitrary, we deduce that there exists $\widehat V\in W$ such that  
\begin{align}
    \langle \partial_t \xhi, \zeta\rangle=(\widehat V,\zeta)_W=\int_0^\infty\int_{\mathbb T^d}  \widehat V\zeta\d\normmm{\mu_t}\dt+\alpha\int_0^\infty\int_{\mathbb T^d} \widehat V\zeta\dx\dt+\alpha\int_0^\infty\int_{\mathbb T^d}\nabla \widehat V\cdot\nabla \zeta\dx\dt\label{base11}
\end{align}
for any $\zeta\in C_c^\infty((0,\infty)\times\om)$. This identity finally explains why the identification \eqref{variation4}, and thus the limitation to dimensions $d=2,3$, is crucial. Indeed, the convergence \eqref{convmeas3} leads to $\mu_t$, and not $\widetilde \mu_t$, which are in principle different, since the only available information is that $\normmm{\mu_t}_{\sf}(\om)=\normmm{\widetilde\mu_t}_{\sf}(\om)$ due to \eqref{identification}, and this is of course not enough to identify the two measures.

 Note that identity \eqref{base11} will become the desired velocity equation when $\alpha\to0$, and that in general $\widehat V$ might be different from the limit potential $u$. The potential $u$ is indeed here only auxiliary and will not appear in the limit as $\alpha\to0$. We now introduce the set $W_{\Ts}$ with the same definition as $W$, but restricted on the time interval $(0,\Ts)$. Then, coming back to \eqref{dtc}, by Cauchy-Schwarz inequality and the lower semicontinuity of the norms we get, for any $\zeta\in C_c^\infty((0,\infty)\times\om)$, 
\begin{align*}
    \normmm{(\widehat V,\zeta)_{W_\Ts}}&=\normmm{\langle \pt \xhi,\zeta\rangle}^2\\&\leq \liminf_{h\to0}\left(\int_0^\Ts \int_{\mathbb T^d} \normmm{u_h(t)}^2\normmm{\rho_h\ast \nabla\chi^h(t)}\dx\dt\right)^\frac12\norm{\zeta}_{L^2(0,\Ts;L^2(\om;\d\normmm{\mu_{(\cdot)}}_{\sf}))}\\&\quad +\alpha^\onehalf\left(\alpha^2\norm{u}_{L^2(0,\Ts;L^2({\mathbb T^d}))}^2\right)^\frac12\alpha^\onehalf\norm{\zeta}_{L^2(0,\Ts;L^2(\om))}\\&\quad +\alpha^\onehalf\left(\alpha^2\norm{\nabla u}_{L^2(0,\Ts;L^2({\mathbb T^d}))}^2\right)^\frac12\alpha^\onehalf\norm{\zeta}_{L^2(0,\Ts;H^1(\om))}\\&\leq \liminf_{h\to0}\left(\int_0^\Ts \int_{\mathbb T^d} \normmm{u_h(t)}^2\normmm{\rho_h\ast \nabla\chi^h(t)}\dx\dt\right)^\frac12\norm{\zeta}_{W_\Ts}\\&\quad +\alpha^\frac12\liminf_{h\to0}\left(\alpha^2\norm{u_h}_{L^2(0,\Ts;L^2({\mathbb T^d}))}^2\right)^\frac12\norm{\zeta}_{W_\Ts}\\&\quad +\alpha^\frac12\liminf_{h\to0}\left(\alpha^2\norm{\nabla u_h}_{L^2(0,\Ts;L^2({\mathbb T^d}))}^2\right)^\frac12\norm{\zeta}_{W_\Ts},
\end{align*}
where we used, recalling again \eqref{convmeas3}, the inequality
\begin{align*}
   & \liminf_{h\to0}\int_0^\Ts \int_{\mathbb T^d} {u_h(t)\zeta(t)}\normmm{\rho_h\ast \nabla\chi^h(t)}\dx\dt\\&\leq \liminf_{h\to0}\left(\left(\int_0^\Ts \int_{\mathbb T^d} \normmm{u_h(t)}^2\normmm{\rho_h\ast \nabla\chi^h(t)}\dx\dt\right)^\frac12\left(\int_0^\Ts \int_{\mathbb T^d} \normmm{\zeta(t)}^2\normmm{\rho_h\ast \nabla\chi^h(t)}\dx\dt\right)^\frac12\right)\\&
    \leq \liminf_{h\to 0}\left(\left(\int_0^\Ts \int_{\mathbb T^d} \normmm{u_h(t)}^2\normmm{\rho_h\ast \nabla\chi^h(t)}\dx\dt\right)^\frac12\left(\int_0^\Ts \int_{\mathbb T^d} \normmm{\zeta(t)}^2\rho_h\ast\normmm{ \nabla\chi^h(t)}\dx\dt\right)^\frac12\right)\\& \leq \liminf_{h\to 0}\left(\int_0^\Ts \int_{\mathbb T^d} \normmm{u_h(t)}^2\normmm{\rho_h\ast \nabla\chi^h(t)}\dx\dt\right)^\frac12\lim_{h\to0}\left(\int_0^\Ts \int_{\mathbb T^d} \normmm{\zeta(t)}^2\rho_h\ast \normmm{\nabla\chi^h(t)}\dx\dt\right)^\frac12\\&  =\liminf_{h\to 0}\left(\int_0^\Ts \int_{\mathbb T^d} \normmm{u_h(t)}^2\normmm{\rho_h\ast \nabla\chi^h(t)}\dx\dt\right)^\frac12\left(\int_0^\Ts \int_{\mathbb T^d} \normmm{\zeta(t)}^2\d\normmm{\mu_t}_{\mathbb S^{d-1}}\dt\right)^\frac12.
\end{align*}
Then, using the definition of $W_\Ts$ and applying Young's inequality, we infer
\begin{align*}
    &\norm{\widehat V}_{W_{\Ts}}^2\\&\leq \left(\liminf_{h\to0}\left(\int_0^\Ts \int_{\mathbb T^d} \normmm{u_h(t)}^2\normmm{\rho_h\ast \nabla\chi^h(t)}\dx\dt\right)^\frac12\right.\\&\qquad\qquad\quad \left.+\alpha^\frac12\liminf_{h\to0}\left(\alpha^2\norm{u_h}_{L^2(0,\Ts;L^2({\mathbb T^d}))}^2\right)^\frac12+\alpha^\frac12\liminf_{h\to0}\left(\alpha^2\norm{\nabla u_h}_{L^2(0,\Ts;L^2({\mathbb T^d}))}^2\right)^\frac12\right)^2\\&
    \leq \liminf_{h\to0}\left(\int_0^\Ts\int_{\mathbb T^d} \normmm{u_h(t)}^2\normmm{\rho_h\ast \nabla\chi^h(t)}\dx\dt\right)\\&
    \quad +2\alpha^\onehalf\liminf_{h\to0}\left(\int_0^\Ts \int_{\mathbb T^d} \normmm{u_h(t)}^2\normmm{\rho_h\ast \nabla\chi^h(t)}\dx\dt\right)^\frac12\\&\quad\times\left(\liminf_{h\to0}\left(\alpha^2\norm{u_h}_{L^2(0,\Ts;L^2({\mathbb T^d}))}^2\right)^\frac12+\liminf_{h\to0}\left(\alpha^2\norm{\nabla u_h}_{L^2(0,\Ts;L^2({\mathbb T^d}))}^2\right)^\frac12\right)\\&\quad +\left(\alpha^\frac12\liminf_{h\to0}\left(\alpha^2\norm{u_h}_{L^2(0,\Ts;L^2({\mathbb T^d}))}^2\right)^\frac12+\alpha^\frac12\liminf_{h\to0}\left(\alpha^2\norm{\nabla u_h}_{L^2(0,\Ts;L^2({\mathbb T^d}))}^2\right)^\frac12\right)^2
    \\&
    \leq (1+\alpha^\frac12)\liminf_{h\to0}\left(\int_0^\Ts\int_{\mathbb T^d} \normmm{u_h(t)}^2\normmm{\rho_h\ast \nabla\chi^h(t)}\dx\dt\right)\\&\quad +(\alpha^\frac12+\alpha)\left(\left(\liminf_{h\to0}\alpha^2\norm{u_h}_{L^2(0,\Ts;L^2({\mathbb T^d}))}^2\right)^\frac12+\left(\liminf_{h\to0}\alpha^2\norm{\nabla u_h}_{L^2(0,\Ts;L^2({\mathbb T^d}))}^2\right)^\frac12\right)^2
    \\&
    \leq (1+\alpha^\frac12)\liminf_{h\to0}\left(\int_0^\Ts\int_{\mathbb T^d} \normmm{u_h(t)}^2\normmm{\rho_h\ast \nabla\chi^h(t)}\dx\dt\right)\\&\quad +2(\alpha^\frac12+\alpha)\liminf_{h\to0}\alpha^2\norm{u_h}_{L^2(0,\Ts;L^2({\mathbb T^d}))}^2+2(\alpha^\frac12+\alpha)\liminf_{h\to0}\alpha^2\norm{\nabla u_h}_{L^2(0,\Ts;L^2({\mathbb T^d}))}^2\\&
    \leq (1+4\alpha^\frac12)\\&\quad \times\liminf_{h\to 0}\left(\int_0^\Ts\int_{\mathbb T^d} \normmm{u_h(t)}^2\normmm{\rho_h\ast \nabla\chi^h(t)}\dx\dt+\alpha^2\norm{u_h}_{L^2(0,\Ts;L^2({\mathbb T^d}))}^2+\alpha^2\norm{\nabla u_h}_{L^2(0,\Ts;L^2({\mathbb T^d}))}^2\right),
\end{align*}
where in the last step we used, just for simplicity, the assumption $\alpha\in(0,1)$.
Clearly, the same argument holds if we use $(s,T')$, for $0<s<T'$ in place of the interval $(0,\Ts)$. This inequality is crucial to obtain the sharp energy inequality in the limit as $\alpha\to0$ in the next section.

\subsubsection{Validity of the energy dissipation inequality.} 
By the results of the previous sections, we can now pass to the limit in \eqref{energy1abc}, first taking $h\to0^+$, then $\tau\to T$ and $\kappa\to s$,  to infer  
\begin{align}
	\nonumber	&E[{\mu_{T'}}]+\onehalf\int_s^{T'} \normmm{\partial E(\mu_t)}_{\mathcal V_{\xhi_t}}^2\dt\\&+\frac1{2(1+4\alpha^\frac12)}\int_s^{T'}\int_{\mathbb T^d} \widehat V(t)^2\d  \normmm{\mu_t}_{\mathbb S^{d-1}}\dt\nonumber+\frac{\alpha}{2(1+4\alpha^\frac12)}\int_{s}^{T'}\left(\normh{\widehat V(t)	}^2+\normh{\nabla \widehat V(t)}^2\right)\dt\\&\nonumber+\frac{\alpha^2}{2}\int_{s}^{T'}(\normh{\tw(t)}^2+\normh{\nabla\tw(t)}^2)\d t\\& \leq E[{\mu_s}]
	\label{energy1abc3b}
\end{align}
for almost any $T'\in (0,\Ts]$, almost any $s\in (0,T')$, and any $\Ts>0$.

The validity of \eqref{energy1abc3b} for $s=0$ can be inferred analogously, where we recall that we have set $E[\mu_t]:=E[\xhi_0]$ when $t=0$.

\subsubsection{Admissibility of the couple $(\xhi,\mu)$.}
\label{admissibilityA}
We now explicitly verify that the couple $(\xhi,\mu)$ is admissible in the sense of Definition \ref{admissibility}. In particular, concerning Property (1), the relation \eqref{PE} can be easily shown by means of the corresponding definition of $\mu_t^h$, the Gau{\ss} Theorem for BV functions and passing to the limit as $h\to 0$, see \cite[Step 12, Proof of item (ii) of Definition 3]{SH}, after a final localization in time. 
 In conclusion, Property (5), i.e., the measurability in time of $\mu_t$ is a direct consequence of its monotonicity.   This allows to conclude that the couple $(\xhi,\mu)$ is admissible in the sense of Definition \ref{admissibility}.  

Note that, to be precise, when $\alpha>0$ many more properties can be shown, see, for instance, the definitions of admissible couple in \cite{SH,AbelsP}. As we will take the limit $\alpha\to 0$ we have checked only the properties that we expect to be preserved in this limit.  

\subsection{Final limit as $\alpha\to0$}
Now that we found the existence of a solution to this kind of modified Mullins-Sekerka system, we can pass to the limit as $\alpha\to0$. This can be obtained similarly as in \cite[Lemma 1]{LH} point (iv)-(Sequential compactness of solution space). In particular, considering a sequence of solutions $\{(\xhi_\alpha,\mu_\alpha)\}_{\alpha}$ indexed by $\alpha>0$, from the energy inequality it is immediate to infer that 
\begin{align}
\sup_{t\geq 0}\normmm{\nabla \xhi_\alpha(t)}(\om)\leq \sup_{t\geq0}\normmm{\mu_{\alpha}^t}_{\sf}(\om)\leq C,\label{energia1}
\end{align}
and that the energy $t\mapsto E[\mu_{\alpha,t}]$ is monotone nonincreasing. In a similar way as for \eqref{convmeas1}, we infer that  there exists ${\wtilde\mu}:{\mathcal B((0,\infty)	\times{\mathbb T^d}\times \mathbb S^{d-1})}\to \mathbb R^+$, such that, up to subsequences, 
\begin{align}
	\mathcal L^1\llcorner{(0,\Ts)}\otimes (\mu_\alpha^t)_{t\in(0,\Ts)}\rightharpoonup {\wtilde\mu}_{\llcorner (0,\Ts)\times \om\times \sf},\quad \text{ weakly* in } \mathcal M((0,\Ts)\times{\mathbb T^d}\times \mathbb S^{d-1}),\label{convmeas12}
\end{align}
 for any $\Ts>0$, where 
\begin{align}
	\wtilde{\mu}=\mathcal L^1\llcorner(0,\infty)\otimes (\mu_t)_{t\in(0,\infty)}.
	\label{representation12}
\end{align} 
Also, as a by product (cf. \eqref{identif}), we can infer that 
\begin{align}
    \lim_{\alpha\to 0} E[\mu_{\alpha}^t]=E[\mu_t],\quad \text{ for a.a. }t>0.\label{Emu}
\end{align}

Now, recalling \eqref{energy1abc3b}, we have that
\begin{align*}
    \sqrt \alpha \norm{\widehat V_\alpha}_{L^2(0,\infty;H^1(\om))}\leq C,
\end{align*}
and thus we immediately deduce 
\begin{align}
    \alpha \widehat V_\alpha\to 0,\text{ strongly in }L^2(0,\infty;H^1({\mathbb T^d})).\label{strongto0}
\end{align}
Also, since by \eqref{energy1abc3b} it holds
$$
\sup_{\alpha>0}\frac1{2(1+4\alpha^\frac12)}\int_0^\infty\int_{\mathbb T^d} \widehat V_\alpha(t)^2\d  \normmm{\mu^{t}_\alpha}_{\mathbb S^{d-1}}\dt\leq E[\xhi_0],
$$
we can thus apply the fundamental \textcolor{black}{result of Hutchinson} \cite[Theorem 4.4.2]{Hutchinson} to infer that the measure-function pairs $(\mathcal L^1{\llcorner}(0,\infty)\otimes\normmm{\mu^t_\alpha}_{\mathbb S^{d-1}})_{t\in(0,\Ts)},\widehat V_\alpha)$ converge in the weak-* sense, up to a further subsequence, to some pair $(\mathcal L^1\llcorner(0,\infty)\otimes (\normmm{\mu_t}_{\mathbb S^{d-1}})_{t\in(0,\Ts)},V(t))$, i.e.,
\begin{align}
  &  \int_0^\infty\int_{\mathbb T^d} (\varphi \widehat V_\alpha)\d\normmm{\mu^t_\alpha}_{\mathbb S^{d-1}} \dt\to \int_0^\infty\int_{\mathbb T^d} \varphi V \d  \normmm{\mu_t}_{\mathbb S^{d-1}}\dt,\label{uhlim}
\end{align}
for any $\varphi\in C_c((0,\infty)\times\om)$, as well as, concerning the dissipative terms, it holds 
\begin{align}
    &\int_s^{T'}\int_{\mathbb T^d} \normmm{V(t)}^2\d  \normmm{\mu_t}_{\mathbb S^{d-1}}\dt\leq \liminf_{\alpha\to0}\int_{s}^{T'}\int_{\mathbb T^d} \normmm{\widehat V_\alpha}^2\d\normmm{\mu^t_\alpha}_{\mathbb S^{d-1}}\dt
    \label{lsc0}
\end{align}
for any $0\leq s<T'$. Since $\frac{1}{2(1+4\alpha^\frac12)}\to \frac12$, we then get from \eqref{lsc0}
\begin{align}
    &\frac12\int_s^{T'}\int_{\mathbb T^d} \normmm{V(t)}^2\d  \normmm{\mu_t}_{\mathbb S^{d-1}}\dt\leq \liminf_{\alpha\to0}\frac1{2(1+4\alpha^\frac12)}\int_s^{T'}\int_{\mathbb T^d} \normmm{\widehat V_\alpha}^2\d\normmm{\mu^t_\alpha}_{\mathbb S^{d-1}}\dt
    \label{lsc}
\end{align}
for any $0\leq s<T'$. The quantity $$V\in L^2((0,\infty)\times {\mathbb T^d};\mathcal L^1{\llcorner}(0,\infty)\otimes (\tmmu)_{t\in(0,\infty)})=L^2(0,\infty;L^2({\mathbb T^d};\normmm{\mu_{(\cdot)}}_{\sf}))$$ is the normal velocity of the interface, in the final limit as $\alpha \to 0$. Indeed, from \eqref{base11}, we have
\begin{align}
\nonumber
   &\langle \pt\xhi_\alpha ,\zeta\rangle=-\int_0^\infty\int_\om \xhi_\alpha\pt\zeta\dx\dt\\&=\int_0^\infty\int_{\mathbb T^d}  \widehat V_\alpha\zeta\d\normmm{\mu_{\alpha}^t}\dt+\alpha\int_0^\infty\int_{\mathbb T^d} \widehat V_\alpha\zeta\dx\dt+\alpha\int_0^\infty\int_{\mathbb T^d}\nabla \widehat V_\alpha\cdot\nabla \zeta\dx\dt\quad \forall \zeta\in C_c^\infty((0,\infty)\times\om),\label{fdt}
\end{align}
which entails from the energy estimate \eqref{energy1abc3b},  extending $\pt\xhi_\alpha$ by density, that 
\begin{align*}
    \norm{\pt \xhi_\alpha}_{L^2(0,\infty;(L^\infty(\om)\cap H^1(\om))')}\leq C.
\end{align*}
This, together with \eqref{energia1} and $\xhi_\alpha\in\{0,1\}$, gives, by Aubin-Lions Lemma,
\begin{align}
    \xhi_\alpha\to \xhi\text{ strongly in }L^p((0,\Ts)\times\om;\{0,1\})\quad\forall \Ts>0 \label{xhialpha}
\end{align}
as $\alpha\to0$. Moreover, arguing as in \eqref{regchi1a}, we get 
$$
\xhi\in C([0,\Ts];L^p(\om))\quad\forall p\in[1,\infty),\quad \forall \Ts>0,
$$
as well as 
$$
\xhi\in L^\infty_{w*}(0,\infty;BV(\om;\{0,1\})).
$$
As a consequence, also recalling the strong convergences \eqref{strongto0}, \eqref{uhlim}, and \eqref{xhialpha}, we can pass to the limit in \eqref{fdt} and deduce 
\begin{align*}
    \langle \pt\xhi,\zeta\rangle=\int_0^\infty \int_\om V\zeta\d\normmm{\mu_t}_{\sf}\dt\quad \forall \zeta\in C_c^\infty((0,\infty)\times\om),
\end{align*}
confirming the interpretation of $V$ as the normal velocity of the evolving interface. This gives $\pt\xhi\in L^2(0,\infty;L^2(\om;\d\normmm{\mu_{(\cdot)}}_{\sf}))$ and thus \eqref{kinetic} holds.

\subsubsection{Limit in the maximal slope} 
We need to control the term related to the first variation of the energy in the energy inequality. Namely, recalling its definition, we have 
\begin{align}\label{slope1}
    \onehalf\int_s^{T'} \normmm{\partial E(\mu^t_\alpha)}_{\mathcal V_{\xhi_\alpha(t)}}^2\dt\geq \int_s^{T'}\delta \mu_\alpha^t(\xi(t))\dt-\frac 12 \int_s^{T'} \norm{\xi(t)}^2_{\mathcal V_{\xhi_\alpha(t)}}\dt,\quad\forall{ \xi\in C^1_c((s,T');S_{\xhi_\alpha{(\cdot)}})},
\end{align}
for any $0\leq s<T'$. We now distinguish two cases.

\textit{(Mean Curvature Flow). }In this case, we have $S_{\xhi_\alpha(t)}=C^1(\om)$ for any $t\geq0$ and thus we do not post process further before passing to the limit as $\alpha\to0$. Thanks to \eqref{convmeas12} we can then let $\alpha \to 0$ and infer 
\begin{align}\label{slope1a}
    \liminf_{\alpha\to0}\onehalf\int_s^{T'} \normmm{\partial E(\mu^t_\alpha)}_{\mathcal V_{\xhi_\alpha(t)}}^2\dt\geq \int_s^{T'}\delta \mu_t(\xi(t))\dt-\frac 12 \int_s^{T'} \norm{\xi(t)}^2_{\mathcal V_{\xhi_t}}\dt,\quad\forall{ \xi\in C^1_c((s,T') \times \om)},
\end{align}
so that, taking the supremum over $\xi\in C^1_c((s,T') \times \om)$, gives 
\begin{align}\label{slope1bb}
    \liminf_{\alpha\to0}\onehalf\int_s^{T'} \normmm{\partial E(\mu^t_\alpha)}_{\mathcal V_{\xhi_\alpha(t)}}^2\dt\geq \onehalf\int_s^{T'} \normmm{\partial E(\mu_t)}_{\mathcal V_{\xhi(t)}}^2\dt
\end{align}
for any $0\leq s<T'$.

\textit{(Volume preserving Mean Curvature Flow).} On the other hand, in the volume preserving case, we need to resort to Lemma \ref{curvature}, with $m_0=\int_\om \xhi_0\dx$ thanks to volume conservation. Indeed, we aim at using $\xi\in C^1_c((s,T');S_{\xhi_{(\cdot)}})$, but in general such $\xi$ does not belong to $C^1_c((s,T');S_{\xhi_\alpha(\cdot)})$ for all $\alpha>0$. Namely, by \eqref{tildeB}, given $\xi\in C^1_c((s,T')\times \om)$, for any $\alpha>0$ we can construct $\wtilde \xi_\alpha\in C^1_c((s,T');S_{\xhi_\alpha(\cdot)})$ such that 
\begin{align*}
&\delta\mu^t_\alpha(\wtilde \xi_\alpha(t))-\norm{\wtilde \xi_\alpha(t)}_{L^2(\om;\d\normmm{\mu_\alpha^t}_{\sf})}^2\\&=\delta\mu^t_\alpha( \xi(t))-\norm{ \xi(t)}_{L^2(\om;\d\normmm{\mu_\alpha^t}_{\sf})}^2+R(\xhi_\alpha(t),\mu_{\alpha}^t,\xi(t))\int_\om \xhi_\alpha(t)\Div \xi(t)\dx,
\end{align*}
where, recalling \eqref{energia1},
\begin{align}
\normmm{R(\xhi_\alpha(t),\mu_\alpha^t,\xi(t))}\nonumber&\leq C(m_0)\left(1+\normmm{\nabla \xhi_\alpha(t)}(\om)\right)\left(1+\norm{\xi(t)}_{C^1(\om)}\right)\normmm{\mu_\alpha^t}_{\sf}(\om)\\&
\leq C(m_0)(1+\norm{\xi(t)}_{C^1(\om)}),\label{Rest}
\end{align}
for any $\alpha>0$.
We can thus reformulate \eqref{slope1} as 
\begin{align}\label{slope1b}\nonumber
   & \onehalf\int_s^{T'} \normmm{\partial E(\mu^t_\alpha)}_{\mathcal V_{\xhi_\alpha(t)}}^2\dt\\&\nonumber\geq \int_s^{T'}\delta \mu_\alpha^t(\xi(t))\dt-\frac 12 \int_s^{T'} \norm{\xi(t)}^2_{\mathcal V_{\xhi_\alpha(t)}}\dt\nonumber\\&\quad +\int_s^{T'} R(\xhi_\alpha(t),\mu_{\alpha}^t,\xi(t))\int_\om \xhi_\alpha(t)\Div \xi(t)\dx\dt, \quad\forall{ \xi\in C^1_c((s,T')\times C^1(\om))},
\end{align}
for any $0\leq s<T'$, and, in particular, we can consider all $\xi \in C^1_c((s,T');S_{\xhi(\cdot)})$. Observe now that, recalling the convergence \eqref{xhialpha} and the definition of $S_{\xhi(t)}$, it holds by Lebesgue's Dominated Convergence that
\begin{align*}
    \int_s^{T'}\normmm{\int_\om \xhi_\alpha(t)\Div \xi(t)\dx}\dt\to 0\quad\text{as }\alpha\to0,\quad \forall \xi \in C^1_c((s,T');S_{\xhi(\cdot)})
\end{align*}
for any $0\leq s<T'$. As a consequence, using \eqref{Rest}, we have
\begin{align*}
   & \normmm{\int_s^{T'} R(\xhi_\alpha(t),\mu_{\alpha}^t,\xi(t))\int_\om \xhi_\alpha(t)\Div \xi(t)\dx\dt}\\&\leq C(m_0)(1+\norm{\xi}_{C^1((s,T')\times\om)})\int_s^{T'}\normmm{\int_\om \xhi_\alpha(t)\Div\xi(t)\dx}\dt\to 0\quad\text{ as }\alpha\to0.
\end{align*}
Therefore, using \eqref{convmeas12}, we can pass to the limit in \eqref{slope1b} obtaining 
\begin{align}\label{slope1bc}
   \nonumber&   \liminf_{\alpha\to 0} \onehalf\int_s^{T'} \normmm{\partial E(\mu^t_\alpha)}_{\mathcal V_{\xhi_\alpha(t)}}^2\dt\\&\geq \int_s^{T'}\delta \mu_t(\xi(t))\dt-\frac 12 \int_s^{T'} \norm{\xi(t)}^2_{\mathcal V_{\xhi_t}}\dt,\quad\forall \xi\in C^1_c((s,T');S_{\xhi_{(\cdot)}}).
\end{align}
Taking now the supremum over $\xi\in C^1_c((s,T');S_{\xhi_{(\cdot)}})$ we finally conclude also in the volume preserving case that 
\begin{align}\label{slopebis}
    \liminf_{\alpha\to0}\onehalf\int_s^{T'} \normmm{\partial E(\mu^t_\alpha)}_{\mathcal V_{\xhi_\alpha(t)}}^2\dt\geq \onehalf\int_s^{T'} \normmm{\partial E(\mu_t)}_{\mathcal V_{\xhi(t)}}^2\dt
\end{align}
for any $0\leq s<T'$.
\subsubsection{Validity of the energy inequality.} 
We have now all the ingredients to pass to the limit in the energy inequality \eqref{energy1abc3b}.
Indeed, thanks to \eqref{Emu}, \eqref{lsc}, and \eqref{slope1bb} (respectively \eqref{slopebis}),  recalling that 
\begin{align*}
  &0\leq \frac{\alpha}{2(1+4\alpha^\frac12)}\int_{s}^{T'}\left(\normh{\widehat V(t)	}^2+\normh{\nabla \widehat V(t)}^2\right)\dt\nonumber\\&\quad \quad +\frac{\alpha^2}{2}\int_{s}^{T'}(\normh{\tw(t)}^2+\normh{\nabla\tw(t)}^2)\d t,
\end{align*}
we can take the liminf as $\alpha\to0$ in the energy inequality  \eqref{energy1abc3b} and finally deduce the validity of \eqref{advectiveMullins}.

This concludes the existence proof of Theorem \ref{thm1}, since the admissibility of the couple $(\xhi,\mu)$ can be verified as in Section \ref{admissibilityA}.

\appendix
\section{A lemma on some norm equivalence}
\begin{lemma}
    \label{Lemmaequiv}
    Let $g\in \mathcal M({\mathbb T^d};\mathbb R^d)$. Then it holds, for any $\alpha>0$,
\begin{align}
\norm{f}_{(L^\infty({\mathbb T^d})\cap H^1({\mathbb T^d}))'}\leq \sqrt 3\sqrt{\max\{ \alpha^2, {\normmm{g}({\mathbb T^d})}\}}\normh{\left({\normmm{\rho_\epsilon\ast g}}+\alpha^2(I-\Delta)\right)^{-\frac12}f},\quad \forall f\in H^1(\om)',
\label{equiv1}
\end{align}
where $\rho_\epsilon$, $\epsilon>0$, is a smooth mollifier.
\end{lemma}
\begin{remark}
Lemma \ref{Lemmaequiv} can be extended also to bounded domains, as long as the Laplacian operator is endowed with suitable boundary conditions.
\end{remark}
\begin{proof}
 First recall the property
    \begin{align*}
        \int_\om \normmm{\rho_\epsilon\ast g}\dx\leq \normmm{g}({\mathbb T^d}),
    \end{align*}
    since ${\mathbb T^d}$ is a
    the flat torus. Then, choosing $\xi\in L^\infty(\om)\cap H^1({\mathbb T^d})$, we get 
    \begin{align*}
        \norm{\sqrt {\normmm{\rho_\epsilon\ast g}}\xi}_{L^2({\mathbb T^d})}\leq \norm{\xi}_{L^\infty({\mathbb T^d})}\normh{\sqrt{\normmm{\rho_\epsilon\ast g}}}=\norm{\xi}_{L^\infty({\mathbb T^d})}\left(\int_\om \normmm{\rho_\epsilon\ast g}\dx \right)^\frac12\leq \sqrt{\normmm{g}({\mathbb T^d})}\norm{\xi}_{L^\infty({\mathbb T^d})}.
    \end{align*}
   Now, given $f\in H^1(\om)'$, let us define $F=(\normmm{\rho_\epsilon\ast g}+\alpha^2(I-\Delta))^{-1}f\in H^1({\mathbb T^d})$. For any $\xi\in L^\infty(\om)\cap H^1({\mathbb T^d})$ we then have
    \begin{align*}
        &\langle f,\xi\rangle_{(H^1({\mathbb T^d})\cap L^\infty(\om))',H^1({\mathbb T^d})\cap L^\infty(\om)}\\&=\int_{\mathbb T^d} (\normmm{\rho_\epsilon\ast g}+\alpha^2)F\xi\dx+\alpha^2\int_{\mathbb T^d} \nabla F\cdot\nabla \xi\dx\\&\leq \normh{\sqrt {\normmm{\rho_\epsilon\ast g}}F}\norm{\sqrt {\normmm{\rho_\epsilon\ast g}}\xi}_{L^2({\mathbb T^d})}+\normh{\alpha F}\normh{ \alpha\xi}+\normh{\alpha\nabla F}\norm{\alpha\xi}_{H^1({\mathbb T^d})}\\&
        \leq \max\{ \alpha, \sqrt{\normmm{g}({\mathbb T^d})}\}\left(\normh{\sqrt {\normmm{\rho_\epsilon\ast g}}F}+\normh{\alpha F}+\normh{\alpha\nabla F}\right)\norm{\xi}_{L^\infty(\om)\cap H^1({\mathbb T^d})}.
    \end{align*}
   Then we get
    $$
    \norm{f}_{(L^\infty(\om)\cap H^1({\mathbb T^d}))'}\leq \max\{\alpha, \sqrt{\normmm{g}({\mathbb T^d})}\}\left(\normh{\sqrt {\normmm{\rho_\epsilon\ast g}}F}+\normh{\alpha F}+\normh{\alpha\nabla F}\right).
    $$
    Taking now the square, and recalling $(a+b+c)^2\leq 3(a^2+b^2+c^2)$, for $a,b,c\geq0$, we deduce 
\begin{align*}
  \norm{f}_{(L^\infty(\om)\cap H^1({\mathbb T^d}))'}^2&\leq \max\{\alpha^2, \normmm{g}({\mathbb T^d})\}\left(\normh{\sqrt {\normmm{\rho_\epsilon\ast g}}F}+\normh{\alpha F}+\normh{\alpha\nabla F}\right)^2\\&\leq 3\max\{\alpha^2, \normmm{g}({\mathbb T^d})\}\left(\normh{\sqrt {\normmm{\rho_\epsilon\ast g}}F}^2+\normh{\alpha F}^2+\normh{\alpha\nabla F}^2\right)
\end{align*}
    concluding the proof as one notices
    \begin{align*}
        \int_{\mathbb T^d} (\normmm{\rho_\epsilon\ast g}+\alpha^2)F^2\dx+\normh{\alpha\nabla F}^2&=\int_{\mathbb T^d} (\normmm{\rho_\epsilon\ast g}+\alpha^2(I-\Delta))F\cdot F\dx \\&=\normh{\left(\normmm{\rho_\epsilon\ast g}+\alpha^2(I-\Delta)\right)^{-\frac12}f}^2.
    \end{align*}
\end{proof}
\section{On the mean curvature of volume preserving Mean Curvature Flow}
Here we present a lemma which allows to show that, given a De Giorgi varifold solution as in Definition \ref{weaksol}, there exists a generalized mean curvature vector satisfying \eqref{PH}. Its proof is an adaptation of similar results in the case of Mullins-Sekerka flow (cf. \cite[Proof of Lemma 3.3]{AbelsRoger}, \cite[Lemma 9]{SH}).
\begin{lemma}
\label{curvature}
    Let $\xhi \in BV({\mathbb T^d};\{0,1\})$ with $m_0:=\int_{\mathbb T^d} \xhi \dx \in(0,\mathcal L^{d}({\om}))$ and let $\mu\in \mathcal M({\mathbb T^d}\times\mathbb S^{d-1})$ be an oriented varifold such that there exists a vector field $\mathbf H^0_{\normmm{\mu}_{\mathbb S^{d-1}}}\in L^2({\mathbb T^d};\d\normmm{\mu}_{\sf};\R^d)$ satisfying
    \begin{align}
\delta\mu(B)=-\int_{{\mathbb T^d}}\mathbf H^0_{\normmm{\mu}_{\mathbb S^{d-1}}}\cdot B\d\normmm{\mu}_{\sf},\quad \forall B\in S_{\xhi},
        \label{divfree}
    \end{align}
    where we recall $S_\xhi=\{B\in C^1({\mathbb T^d}):\ \int_{\mathbb T^d} \xhi\Div B\dx=0\}$.

Then there exists $\lambda\in \R$ such that 
\begin{align}
    \normmm{\lambda}\leq C(\om,m_0)(1+\normmm{\nabla \xhi}({\mathbb T^d}))\left(\normmm{\mu}_{\sf}({\mathbb T^d})+\sqrt{\normmm{\mu}_{\sf}({\mathbb T^d})}\norm{\mathbf H^0_{\normmm{\mu}_{\sf}}}_{L^2(\om;\d\normmm{\mu}_{\sf}({\mathbb T^d}))}\right),\label{ctrl1}
\end{align}
for some $C(\om,m_0)>0$ independent of $\xhi$ and $\mu$, and such that
\begin{align}
    \delta\mu(B)=-\int_{\mathbb T^d} \mathbf H_{\normmm{\mu}_{\sf}}\cdot B\d\normmm{\mu}_{\sf},\quad \forall B\in C^1(\om),
    \label{totvar}
\end{align}
where
\begin{align*}
    \mathbf H_{\normmm{\mu}_{\sf}}(x):=   \mathbf H_{\normmm{\mu}_{\sf}}^0(x)+\lambda\int_{\om} B\cdot\d \nabla \xhi, \text{ for }\normmm{\mu}_{\sf}\text{-a.a. }x\in {\mathbb T^d}.
\end{align*}
Moreover, for any $B\in C^1(\om)$ there exists $\widetilde B\in S_\xhi$ such that 
\begin{align}
    \delta\mu(\wtilde B)-\Vert{\wtilde B}\Vert^2_{L^2(\om;\d\normmm{\mu}_{\sf})}=\delta\mu(B)-\Vert{ B}\Vert^2_{L^2(\om;\d\normmm{\mu}_{\sf})}+R(\xhi,\mu,B)\int_\om \xhi\Div B\dx,\label{tildeB}
\end{align}
where 
\begin{align*}
    \normmm{R(\xhi,\mu,B)}&\leq C(\om,m_0)(1+\norm{B}_{C^1(\om)})(1+\normmm{\nabla \xhi}(\om))\normmm{\mu}_{\sf}(\om).
\end{align*}
\end{lemma}
\begin{proof}
    We can argue as in \cite{AbelsRoger}, namely, let us fix $\xi\in C^1({{\mathbb T^d}})$ to be chosen later on. Then, given $B\in C^1({\mathbb T^d})$, we consider the vector field 
    \begin{align}
    \widetilde B:=B-\frac{\int_{\mathbb T^d} \xhi\Div B\dx}{\int_{\mathbb T^d} \xhi\Div \xi\dx}\xi\in S_\chi.
    \label{tB}
    \end{align}
    Since $\wtilde B\in S_\xhi$, we can apply \eqref{divfree} to obtain, after some trivial algebraic manipulations,
\begin{align*}
    \delta\mu(B)=-\int_{\mathbb T^d} \mathbf H_{\normmm{\mu}_{\sf}}^0\cdot B\d\normmm{\mu}_{\sf}+\frac{\int_{\mathbb T^d} \xhi\Div B\dx}{\int_{\mathbb T^d} \xhi\Div \xi\dx}\left(\delta\mu(\xi)+\int_{\mathbb T^d} \mathbf H_{\normmm{\mu}_{\sf}}^0\cdot \xi\d\normmm{\mu}_{\sf}\right).
\end{align*}
We now define $\lambda\in \R$ as
\begin{align}
 \lambda:=\frac{\delta\mu(\xi)+\int_{\mathbb T^d} \mathbf H_{\normmm{\mu}_{\sf}}^0\cdot \xi\d\normmm{\mu}_{\sf}}{\int_{\mathbb T^d} \xhi\Div \xi\dx},\label{lambdaw}
\end{align}
so that we deduce
\begin{align*}
    \delta\mu(B)&=-\int_{\mathbb T^d} \mathbf H_{\normmm{\mu}_{\sf}}^0\cdot B\d\normmm{\mu}_{\sf}+\lambda\int_{\mathbb T^d} \xhi\Div B\dx\\&
    =-\int_{\mathbb T^d} \mathbf H_{\normmm{\mu}_{\sf}}^0\cdot B\d\normmm{\mu}_{\sf}-\lambda\int_{\mathbb T^d} B\cdot \d\nabla \xhi,
\end{align*}
    giving \eqref{totvar}. It is then enough to choose a vector field $\xi$ so that \eqref{ctrl1} holds. We can follow the same choice as in the proof of \cite[Lemma 3.3]{AbelsRoger} (see also \cite[Lemma 9]{SH}), so that it holds 
    \begin{align}
    \int_{\mathbb T^d} \xhi \Div \xi\dx\geq \left(1-\frac{m_0}{\mathcal L^d({\mathbb T^d})}\right)m_0\mathcal{L}^d({\mathbb T^d})-C({\mathbb T^d})\epsilon(1+\normmm{\nabla \xhi}({\mathbb T^d}))\geq C({\mathbb T^d},m_0)>0,
    \label{lower}\end{align}
    where we can choose $\epsilon:=\frac{\left(1-\frac{m_0}{\mathcal L^d({\mathbb T^d})}\right)m_0\mathcal{L}^d({\mathbb T^d})}{4C({\mathbb T^d})(1+\normmm{\nabla \xhi}({\mathbb T^d}))}$. Also, the choice of $\xi$ gives the control (cf. \cite[(3.27)]{AbelsRoger})
    \begin{align}
        \norm{\xi}_{C^1({\mathbb T^d})}\leq \frac{C(\om)}{\epsilon},\label{ctr}
    \end{align}
so that, coming back to \eqref{lambdaw}, we have, by Cauchy-Schwarz inequality and recalling the definition of $\delta\mu$, that 
\begin{align*}
    \normmm{\lambda}&\leq C(\om,m_0)\left(\norm{\xi}_{C^1(\om)}\normmm{\mu}_{\sf}(\om)+\norm{\xi}_{C(\om)}\sqrt{\normmm{\mu}_{\sf}(\om)}\norm{\mathbf H^0_{\normmm{\mu}_{\sf}}}_{L^2(\om;\d \normmm{\mu}_{\sf};\R^d)}\right)\\&
    \leq \frac{C(\om,m_0)}{\epsilon}\left(\normmm{\mu}_{\sf}(\om)+\sqrt{\normmm{\mu}_{\sf}(\om)}\norm{\mathbf H^0_{\normmm{\mu}_{\sf}}}_{L^2(\om;\d \normmm{\mu}_{\sf};\R^d)}\right),
\end{align*}
giving \eqref{ctrl1} after recalling the choice of the value $\epsilon>0$ above. 

In conclusion, as a byproduct we can prove \eqref{tildeB}. Indeed, it is enough, given $B\in C^1(\om)$, to define $\wtilde B\in S_\xhi$ as in \eqref{tB}, so that it holds   
\begin{align*}
     &\delta\mu(\wtilde B)-\Vert{\wtilde B}\Vert^2_{L^2(\om;\d\normmm{\mu}_{\sf})}\\&
     =\delta\mu(B)-\Vert{ B}\Vert^2_{L^2(\om;\d\normmm{\mu}_{\sf})}-\frac{\int_{\mathbb T^d} \xhi\Div B\dx}{\int_{\mathbb T^d} \xhi\Div \xi\dx}\delta\mu(\xi)-\left(\frac{\int_{\mathbb T^d} \xhi\Div B\dx}{\int_{\mathbb T^d} \xhi\Div \xi\dx}\right)^2\norm{\xi}^2_{L^2(\om;\d\normmm{\mu}_{\sf})} \\&
     \quad +2\frac{\int_{\mathbb T^d} \xhi\Div B\dx}{\int_{\mathbb T^d} \xhi\Div \xi\dx}\int_\om \xi\cdot B\d\normmm{\mu}_{\sf}:= \delta\mu(B)-\Vert{ B}\Vert^2_{L^2(\om;\d\normmm{\mu}_{\sf})}+R(\xhi,\mu,B)\int_\om \xhi\Div B\dx .
\end{align*}
Then, recalling \eqref{lower} and \eqref{ctr}, we can estimate
\begin{align*}
    &\normmm{R(\chi,\mu,B)}\\&\leq C(\om,m_0)\left(\norm{\xi}_{C^1(\om)}\normmm{\mu}_{\sf}(\om)+\norm{B}_{C^1(\om)}\mathcal L^{d}({\om})\norm{\xi}_{C(\om)}\normmm{\mu}_{\sf}(\om)\right)\\&\quad +C(\om,m_0)\left(2\mathcal L^{d}({\om})\norm{\xi}_{C(\om)}\norm{B}_{C(\om)}\normmm{\mu}_{\sf}(\om)\right)\\&
    \leq C(\om,m_0)(1+\norm{B}_{C^1(\om)})(1+\normmm{\nabla \xhi}(\om))\normmm{\mu}_{\sf}(\om),
\end{align*}
which gives \eqref{tildeB} and concludes the proof of the lemma.

\end{proof}
\textbf{Acknowledgments.} The authors warmly thank Sebastian Hensel for some useful comments on a preliminary version of the manuscript, especially about the definition of \eqref{slope} and Remark \ref{meancurvature}. AP also gratefully acknowledges support from the Alexander von Humboldt Foundation and is a member of Gruppo Nazionale per l’Analisi Matematica, la Probabilità e le loro Applicazioni (GNAMPA) of
Istituto Nazionale per l’Alta Matematica (INdAM).  

	\bibliographystyle{siam}
	\bibliography{Bib}
	
\end{document}